\documentclass[11pt]{article}
\usepackage{amssymb}
\usepackage{amsmath}
\usepackage{amsthm}
\usepackage{yhmath}
\usepackage{mathdots}
\usepackage{MnSymbol}
\usepackage{a4wide}
\usepackage{color}
\usepackage[left=2.5cm,top=2cm,right=2cm]{geometry}
\usepackage{enumitem}
\usepackage{nccmath}
\usepackage{booktabs}
\usepackage[square,numbers]{natbib}
\usepackage{graphicx}
\graphicspath{ {/Users/neha/Desktop/NLA P1/paper1 new} }
\usepackage{subfigure}

\usepackage{natbib}
\usepackage[english]{babel}
\usepackage{algorithm}
\usepackage{algpseudocode}
\newtheorem{problem}{\bf Problem}[section]
\newtheorem{theorem}{\bf Theorem}[section]
\newtheorem{definition}[theorem]{\bf Definition}

\newtheorem{exam}[theorem]{\bf Example}
\newtheorem{remark}[theorem]{\bf Remark}
\newtheorem{lemma}[theorem]{\bf Lemma}

\SetLabelAlign{CenterWithParen}{\hfil(\makebox[1.0em]{#1})\hfil}
\def \R{{\mathbb R}}
\def \C{{\mathbb C}}
\def \rank{\mathrm{rank}}
\def \vec{\mathrm{vec}}
\def \QR{\mathbb{Q}_\mathbb{R}}
\def \i{\textit{\textbf{i}}}
\def \j{\textit{\textbf{j}}}
\def \k{\textit{\textbf{k}}}
\newcommand\norm[1]{\left\lVert#1\right\rVert}
\newcommand{\Rn}[1]{%
	\textup{\lowercase\expandafter{\romannumeral#1}}%
}
\def \ov{\overrightarrow}
\def \diag{\mathrm{diag}}

\def \noin{\noindent}
\newcommand{\beano}{\begin{eqnarray*}}	\newcommand{\eeano}{\end{eqnarray*}}
\renewcommand{\thefootnote}{\fnsymbol{footnote}}

\date{\today}
\title{L-structure least squares solutions of reduced biquaternion matrix equations with applications 
	
	\footnotemark[2]}
\author{Sk. Safique Ahmad\footnotemark[1] \and Neha Bhadala \footnotemark[3]}

\begin{document}
	\maketitle
	\begin{abstract}
		This paper presents a framework for computing the structure-constrained least squares solutions to the generalized reduced biquaternion matrix equations (RBMEs). The investigation focuses on three different matrix equations: a linear matrix equation with multiple unknown L-structures, a linear matrix equation with one unknown L-structure, and the general coupled linear matrix equations with one unknown L-structure. Our approach leverages the complex representation of reduced biquaternion matrices. To showcase the versatility of the developed framework, we utilize it to find structure-constrained solutions for complex and real matrix equations, broadening its applicability to various inverse problems. Specifically, we explore its utility in addressing partially described inverse eigenvalue problems (PDIEPs) and generalized PDIEPs. Our study concludes with  numerical examples. 
	\end{abstract}
	
	
	
	\renewcommand{\thefootnote}{\fnsymbol{footnote}}
	
	\footnotetext[1]{
		Department of Mathematics, Indian Institute of Technology Indore, Simrol, Indore-452020, Madhya Pradesh, India. \texttt{email:safique@iiti.ac.in, safique@gmail.com}.}
	
	\footnotetext[3]{Research Scholar, Department of Mathematics, IIT Indore, Research work funded by PMRF (Prime Minister's Research Fellowship). \texttt{email:phd1901141004@iiti.ac.in, bhadalaneha@gmail.com}}
	
	\section{Introduction} \label{Intro}
	In matrix theory, linear matrix equations play a crucial role. They appear widely in control theory, inverse problems, and linear optimal control \cite{ datta2004numerical, gantmacher2005applications, jameson1973inverse}. Owing to their widespread application in various fields, one encounters the problem of finding approximate solutions for linear matrix equations. There are many different forms of matrix equations. Some simple examples of these are:
	\begin{equation*}
		AX= B, \; \; \; AXB+ CX^TD= E, \; \; \; AXB+ CYD= E.
	\end{equation*}
	There have been several studies on real and complex matrix equations. See \cite{chu1987singular, flanders1977matrix, jiang2003solutions, magnus1983structured} and references therein. We now turn our attention to quaternion and reduced biquaternion matrix equations. 
	
	In $1843$, Hamilton first introduced the notion of quaternions. Several aspects of quantum physics, image processing, and signal processing rely on the quaternion matrix equations and their general solutions \cite{ bulow2001hypercomplex, yuan2013solutions, zhang2015special}. Quaternion matrix equations have been the subject of several studies in the literature (for example, \cite{liping1996matrix, yuan2016structured, yuan2011least}). Quaternion multiplication is not commutative, which limits its use in many vital applications.  Following Hamilton's discovery of quaternions, Segre introduced reduced biquaternions, which are commutative in nature. Reduced biquaternions are also known as commutative quaternions. Commutative multiplication simplifies numerous operations. For instance, Pei et al. \cite{pei2004commutative, pei2008eigenvalues} demonstrated how reduced biquaternions outperform conventional quaternions in image and digital signal processing. Additionally, reduced biquaternions allow us to treat three or four-dimensional vectors as one entity, facilitating efficient information processing. Due to this, it becomes imperative to learn how to solve the matrix equations arising from commutative quaternionic theory. Some studies in the literature have focused on reduced biquaternion matrix equations (RBMEs). For instance, Yuan et al. \cite{yuan2020hermitian} discussed the Hermitian solution of the RBME $(AXB, CXD)= (E, G)$. Zhang et al. \cite{zhang2020algebraic} investigated the least squares solutions to the matrix equations $AXC= B$ and $AX= B$. This paper focuses on the least squares structured solutions to the generalized RBMEs. Our framework encompasses all those structures where any set of linear relationships between the matrix entries is allowed. A class of such matrices is called reduced biquaternion L-structure. Surprisingly, the least squares Toeplitz, symmetric Toeplitz, Hankel, and circulant solutions of the RBME have not been discussed in the literature despite their significance in scientific computing, inverse problems, image restoration, and signal processing  \cite{carrasquinha2018image, nandhini2019compressive,zhang2019correction}. Given the above context, our interest lies in least squares L-structure solutions for generalized RBMEs, with specific attention to reduced biquaternion Toeplitz, symmetric Toeplitz, Hankel, and circulant solutions. This manuscript addresses the following generalized matrix equations: 
\begin{eqnarray} 
		&&\sum_{l=1}^{r}A_lX_lB_l= E,    \label{eq2}\\ 
	    &&	\sum_{l=1}^{r}A_lXB_l+ \sum_{p=1}^{q}C_pX^TD_p= E , \label{eq1}   \\ 
         && 	(A_1XB_1, A_2XB_2, \ldots, A_rXB_r)= (E_1, E_2, \ldots, E_r).  \label{eq3}          
\end{eqnarray}    
Moreover, this paper elucidates a range of applications for the proposed framework in solving inverse eigenvalue problems. Several applications of the inverse eigenvalue problem, which involves reconstructing matrices from prescribed spectral data, deal with structured matrices. When the spectral data contain only partial information about the eigenpairs, this kind of inverse problem is called a partially described inverse eigenvalue problem (PDIEP). In  both PDIEP and generalized PDIEP, two pivotal questions arise: the theory of solvability and the numerical solution methodology (see textbook \cite{chu2005inverse} and references therein). In the context of solvability, one major challenge has been identifying necessary or sufficient conditions for a PDIEP and a generalized PDIEP to be solvable. On the other hand, numerical solution methods aim to develop procedures that construct a matrix in a numerically stable manner when the given spectral data is feasible. In this paper, we have successfully developed a numerical solution methodology for both PDIEP and generalized PDIEP by employing our developed framework. Our attention is primarily directed toward two structures, namely Hankel and symmetric Toeplitz. In summary, the primary applications discussed in this article encompass:
	\begin{itemize}
	    \item We utilize our developed framework to determine structure-constrained solutions for complex and real matrix equations. This is possible because these matrix equations represent a special case of RBMEs. It enables us to tackle various inverse eigenvalue problems. Using our framework, we offer a solution to PDIEP \cite[Problem 5.1 and  5.2]{chu2005inverse}, which involves constructing a structured matrix from an eigenpair set.
	    \item We provide a framework for solving generalized PDIEP for symmetric Toeplitz and Hankel structure.
	\end{itemize}
	
	The manuscript is organized as follows. Section \ref{sec2} presents the notation and preliminary results. In Section \ref{sec3}, we first define the concept of reduced biquaternion L-structures and examine some reduced biquaternion L-structure matrices. Next, we present several useful lemmas. Section \ref{sec4} outlines the general framework for solving the RBMEs. Subsection \ref{sec4.2} delves into solving the RBME with multiple unknown L-structures. In Section \ref{sec5}, we apply the developed framework from Section \ref{sec4} to specific cases and explore their practical implications. Finally, Section \ref{sec6} provides the numerical verification of our developed results. 
	\section{Notation and preliminaries} \label{sec2}
	\subsection{Notation}  \label{sec2.1}
	Throughout this paper, we denote $\QR$ as the set of all reduced biquaternions. $\R^{m \times n}$, $\C^{m \times n}$, and $\QR^{m \times n}$ denote the sets of all $m \times n$ real, complex, and reduced biquaternion matrices, respectively. Denote $\mathbb{T}\R^{n \times n}$, $\mathbb{S}\mathbb{T}\R^{n \times n}$, $\mathbb{H}\R^{n \times n}$, $\C\R^{n \times n}$, $\mathbb{H}\C^{n \times n}$, $\mathbb{T}\QR^{n \times n}$, $\mathbb{S}\mathbb{T}\QR^{n \times n}$, $\mathbb{H}\QR^{n \times n}$, and $\C\QR^{n \times n}$ as the sets of all $n \times n$ real Toeplitz, real symmetric Toeplitz, real Hankel, real circulant, complex Hankel, reduced biquaternion Toeplitz, reduced biquaternion symmetric Toeplitz, reduced biquaternion Hankel, and reduced biquaternion circulant matrices, respectively. For $A \in \R^{m \times n}$, the notations $A^{+}$ and $A^T$ denote the Moore-Penrose generalized inverse and the transpose of $A$. For $A \in \C^{m \times n}$, the notations $\Re(A)$ and $\Im(A)$ stand for the real and imaginary parts of $A$, respectively. For a diagonal matrix $A=(a_{ij}) \in \QR^{n \times n}$, we denote it as $\diag(\alpha_1,\alpha_2,\ldots,\alpha_n)$, where $a_{ij}=0$ whenever $i \neq j$ and $a_{ii}=\alpha_i$ for $i=1,\ldots,n$. $I_{n}$ represents the identity matrix of order $n$. For $i=1, 2, \ldots, n$, $e_i$ denotes the $ith$ column of the identity matrix $I_n$. $0$ denotes the zero matrix of suitable size. $A \otimes B= (a_{ij}B)$ represents the Kronecker product of matrices $A \; \mbox{and} \; B$. For matrix $A= (a_{ij}) \in \QR^{m \times n}$, $\vec(A)= \left[a_1, a_2, \ldots, a_n\right]^T$, where $a_{j}=\left[a_{1j}, a_{2j}, \ldots, a_{mj}\right] \; \mbox{for} \; j= 1,  2, \ldots, n$. $\norm{\cdot}_F$ represents the Frobenius norm. $\norm{\cdot}_2$ represents the $2$-norm or Euclidean norm. For $A \in \QR^{m \times n_1}$ and $B \in \QR^{m \times n_2}$, the notation $\left[A, B\right]$ represents the matrix $\begin{bmatrix}A & B\end{bmatrix} \in \QR^{m \times (n_1 + n_2)}$.
	
	The Matlab command $rand(m, n)$ and $ones(m, n)$ return an $m\times n$ random matrix and a matrix with all entries one, respectively. Let $u$ and $v$ be row vectors of size $n$. Matlab command $toeplitz(u, v)$ returns a Toeplitz matrix with $u$ as its first column and $v$ as its first row. Matlab command $toeplitz(u)$ returns a symmetric Toeplitz matrix with $u$ as its first column and first row. Matlab command $hankel(u, v)$ creates a Hankel matrix with $u$ and $v$ as its first column and last row, respectively.
	 We use the following abbreviations throughout this paper:\\
	RBME : reduced biquaternion matrix equation,
	PDIEP : partially described inverse eigenvalue problem.
	\subsection{Preliminaries}  \label{sec2.2}
	A reduced biquaternion can be expressed uniquely as  $r= r_{0}+ r_{1}\i+ r_{2}\j+ r_{3}\k$,
	where $r_{i} \in \R$ for $i= 0, 1, 2, 3$, and $\i^2=  \k^2= -1, \; \j^2= 1$,
	$\i\j= \j\i= \k, \; \j\k= \k\j= \i, \; \k\i= \i\k= -\j$. It can also be expressed as $r= d_1+ d_2\j$, where $d_1= r_0+ r_1\i$ and $d_2= r_2+ r_3\i$ are complex numbers. The conjugate of $r$, denoted as $\bar{r}$, is given by $\bar{r}= r_0- r_1\i -r_2\j- r_3\k$. The norm of $r$ is $\norm{r}= \sqrt{r_0^2+ r_1^2+ r_2^2+ r_3^2}.$
	We have
	\[
	\|r\|^2 \neq r \bar{r}.
	\]
	In addition, we identify $r= d_1+ d_2\j \in \QR$ using a complex vector $\Psi_{r}= [d_1, d_2] \in \C^{1 \times 2}$. Similarly, we identify any reduced biquaternion matrix $Z= Z_1+ Z_2\j \in \QR^{m \times n}$, where $Z_1, Z_2 \in \C^{m \times n}$, using a complex matrix $\Psi_Z= \left[Z_1, Z_2\right]\in \C^{m \times 2n}$. The Frobenius norm for $Z= (z_{ij}) \in \QR^{m \times n}$ is defined as follows:
	\[
	\norm{Z}_F= \sqrt{\sum_{i= 1}^{m} \sum_{j= 1}^{n} \norm{z_{ij}}^{2}}.
	\]
	We have
	\[
	\norm{Z}_F= \norm{\Psi_{Z}}_F= \sqrt{\norm{Z_1}_F^2+ \norm{Z_2}_F^2}= \sqrt{\norm{\Re(Z_1)}_F^2+ \norm{\Im(Z_1)}_F^2+ \norm{\Re(Z_2)}_F^2+ \norm{\Im(Z_2)}_F^2}.
	\]
	The complex representation of matrix $Z= Z_1+ Z_2\j \in \QR^{m \times n}$, denoted as $h(Z)$, is defined as follows:
	\begin{equation*} 
		h(Z)=\begin{bmatrix}
			Z_1 & Z_2\\
			Z_2 & Z_1                    
		\end{bmatrix}.                         
	\end{equation*}
	For $Y \in \QR^{m \times n}$ and $Z \in \QR^{n \times p}$, we have
	\begin{equation} \label{eq4}
		h(YZ)= h(Y)h(Z).
	\end{equation}
	For $q= q_1+ q_2\j \in \QR$ and $Y= Y_1+ Y_2\j \in \QR^{m \times n}$, $\Psi_{qY}$ can be expressed as
	\begin{equation*}
		\Psi_{qY}= \left[q_1Y_1+ q_2Y_2, q_1Y_2+ q_2Y_1\right]= \left[q_1,q_2\right]\begin{bmatrix}
			Y_1 & Y_2\\
			Y_2 & Y_1
		\end{bmatrix}= \Psi_{q}h(Y).
	\end{equation*}
	For $\alpha \in \R$, $Y= Y_1+ Y_2\j$, and $Z= Z_1+ Z_2\j$, we have  $\Psi_{\alpha Y}= \alpha \Psi_Y$, $\Psi_{Y+ Z}= \Psi_Y+ \Psi_Z$,
	and
	\begin{equation}\label{eq5}
		\Psi_{YZ}= \left[Y_1Z_1+ Y_2Z_2, Y_1Z_2+ Y_2Z_1\right]= \left[Y_1, Y_2\right] \begin{bmatrix}
			Z_1 & Z_2\\
			Z_2 & Z_1
		\end{bmatrix}= \Psi_{Y}h(Z).
	\end{equation}
	Furthermore, the operator $\vec(Y)$ is linear, which means that
	$\vec(Y+Z)= \vec(Y)+ \vec(Z) \,\mbox{and}\,\vec(\alpha Y)= \alpha \vec(Y).$
	For $Z= Z_1+ Z_2\j$, we have
	\begin{equation}	\label{eq6}
		\vec(Z)= \vec(Z_1)+ \vec(Z_2)\j \; \;  \mbox{and} \;
		\vec(\Psi_Z)= 
		\begin{bmatrix}
			\vec(Z_1) \\ 
			\vec(Z_2)
		\end{bmatrix}.
	\end{equation}
	Let $Z=Z_1+ Z_2\j \in \QR^{m \times n}$ and denote $\ov{Z}= \left[\Re(Z_1), \Im(Z_1), \Re(Z_2), \Im(Z_2)\right] \in \R^{m \times 4n}$. We have
	$$
	\vec(\ov{Z})= \begin{bmatrix}
		\vec(\Re(Z_1)) \\
		\vec(\Im(Z_1)) \\
		\vec(\Re(Z_2)) \\
		\vec(\Im(Z_2))
	\end{bmatrix}. $$
	Clearly,
	\begin{equation} \label{eq7}
		\norm{Z}_F= \norm{\Psi_Z}_F= \norm{\vec(\Psi_Z)}_F= \norm{\vec(\ov{Z})}_F.
	\end{equation}
	\section{Reduced biquaternion L-structure matrices} \label{sec3}
	This section aims to define the concept of reduced biquaternion L-structure and explore some specific examples of this class of matrices. A reduced biquaternion L-structure refers to the set of all reduced biquaternion matrices of a given order whose entries adhere to specific linear constraints. A notable example of this class includes unstructured matrices, where no linear restrictions are placed on the matrix entries. The subsequent definition offers a formalized explanation of this concept.
	\begin{definition}
		Let $\Omega$ be a subspace of $\QR^{mn}$. The subset of reduced biquaternion matrices of order $m \times n$ given by
		\begin{equation} \label{eq8}
			L(m,n)= \{X \in \QR^{m \times n} | \vec(X) \in \Omega\}
		\end{equation}
		is known as the reduced biquaternion L-structure. 
	\end{definition}
	\begin{remark}
		$\QR$ and $\QR^n$ are vector spaces over $\R$ with dimensions $4$ and $4n$, respectively.
	\end{remark}
	To better comprehend the above definition, let us consider the following examples.  
	\begin{exam}
		Let $A= \begin{bmatrix}
			0 & 0 & 0 & 1 & 0 & 0 & 0 & 0 & 0\\
			0 & 0 & 0 & 0 & 0 & 0 & 1 & 0 & 0\\
			0 & 1 & 0 & 0 & 0 & 0 & 0 & 0 & 0\\
			0 & 0 & 0 & 0 & 0 & 0 & 0 & 1 & 0\\
			0 & 0 & 1 & 0 & 0 & 0 & 0 & 0 & 0\\
			0 & 0 & 0 & 0 & 0 & 1 & 0 & 0 & 0\\
		\end{bmatrix}$ and $\Omega_{1}= \{v \in \QR^{9 \times 1} \; | \; Av= 0\}$. Clearly, $\Omega_{1}$ is a subspace of $\QR^{9 \times 1}$. The resulting reduced biquaternion L-structure is as follows:
		$$L(3,3)= \{X \in \QR^{3 \times 3} \; | \; \vec(X) \in \Omega_{1}\}.$$
		The subset above represents the class of diagonal matrices of size $3 \times 3$. In this case, six linear restrictions are imposed on the entries of matrix $X= (x_{ij}) \in \QR^{3 \times 3}$, given by $x_{ij}= 0$ for $i \neq j$. Hence, the collection of all reduced biquaternion diagonal matrices of a given order falls under the class of reduced biquaternion L-structure.
	\end{exam}   
 Other reduced biquaternion L-structure examples include the set of all reduced biquaternion Toeplitz, symmetric Toeplitz, Hankel, circulant, lower triangular, and upper triangular matrices of a given order. These classes of matrices consider only equality relationships between the matrix entries. Here is an example of a reduced biquaternion L-structure with some linear relationships between the matrix entries.
	\begin{exam}
		Let $B= \begin{bmatrix}
			1 & -1 & 1 & 0 & 0 & 0 & 0 & 0 & 0\\
			0 & 0 & 0 & 1 & 1 & -1 & 0  &  0 & 0\\
			0 & 0 & 0 & 0 & 0 & 0 & 1 & -1 & 1
		\end{bmatrix}$ and $\Omega_{2}= \{v \in \QR^{9 \times 1} \; | \; Bv= 0\}$. Clearly, $\Omega_{2}$ is a subspace of $\QR^{9 \times 1}$. The resulting reduced biquaternion L-structure is as follows:
		$$L(3,3)= \{X \in \QR^{3 \times 3} \; | \; \vec(X) \in \Omega_{2}\}.$$
		The above subset represents a collection of all reduced biquaternion matrices $X= (x_{ij}) \in \QR^{3 \times 3}$ with the following linear restrictions imposed on the entries of matrix $X$: $x_{11}+ x_{31}= x_{21}$, $x_{12}+ x_{22}= x_{32}$, and $x_{13}+ x_{33}= x_{23}$.
	\end{exam}   
	\noin The remaining section focuses on some reduced biquaternion L-structure matrices that frequently appear in practical applications. Our primary focus lies on  {\it reduced biquaternion Toeplitz, symmetric Toeplitz, Hankel, and circulant matrices.} To commence our exploration, we initially examine the vec-structure of some real structured matrices. 
	\begin{definition}
		A matrix $X \in \R^{n \times n}$ is Toeplitz if it has the following form:
		\[
		X=  \begin{bmatrix}
			x_{0}      & x_{1}   & x_{2}  &\cdots & \cdots  & x_{n-1} \\
			x_{-1}     & x_{0}  & x_{1}   &\ddots &             & \vdots  \\
			x_{-2}    & x_{-1} &\ddots  &\ddots &\ddots  & \vdots  \\
			\vdots    & \ddots&\ddots  &\ddots & x_{1}    & x_{2}   \\
			\vdots    &            &\ddots  &x_{-1}  & x_{0}   & x_{1}   \\
			x_{-n+1}& \cdots&\cdots  &x_{-2}  & x_{-1}  & x_{0}
		\end{bmatrix}.
		\]
		For $X \in \R^{n \times n}$, denote by $\vec_T(X)$ the following vector:
		\begin{equation}\label{eq9}
			\vec_T(X)= \left[x_{-n+ 1}, x_{-n+ 2}, \ldots, x_{-1}, x_{0}, x_{1}, \ldots, x_{n-1}\right]^T \in \R^{2n-1}.     
		\end{equation}
	\end{definition} 
	\begin{definition}
	A matrix $X \in \R^{n \times n}$ is symmetric Toeplitz if it has the following form:
	\[
	X=  \begin{bmatrix}
		x_{0}      & x_{1}   & x_{2}  &\cdots & \cdots  & x_{n-1} \\
		x_{1}     & x_{0}  & x_{1}   &\ddots &             & \vdots  \\
		x_{2}    & x_{1} &\ddots  &\ddots &\ddots  & \vdots  \\
		\vdots    & \ddots&\ddots  &\ddots & x_{1}    & x_{2}   \\
		\vdots    &            &\ddots  &x_{1}  & x_{0}   & x_{1}   \\
		x_{n-1}& \cdots&\cdots  &x_{2}  & x_{1}  & x_{0}
	\end{bmatrix}.
	\]
	For $X \in \R^{n \times n}$, denote by $\vec_{ST}(X)$ the following vector:
	\begin{equation}\label{eq3.1}
		\vec_{ST}(X)= \left[x_{0}, x_{1}, x_{2}, \ldots, x_{n-1}\right]^T \in \R^{n}.
	\end{equation}
\end{definition} 
	\begin{definition}
		A matrix $X \in \R^{n \times n}$ is Hankel if it has the following form:
		\[
		X=  \begin{bmatrix}
			x_{n-1} & \cdots& \cdots &x_{2}  & x_{1}   & x_{0} \\
			\vdots  &            &\udots  &x_{1 } & x_{0}   & x_{-1} \\
			\vdots  & \udots&\udots  &\udots & x_{-1}  & x_{-2} \\
			x_{2}   & x_{1}   &\udots  &\udots & \udots  & \vdots  \\
			x_{1}   & x_{0}   &x_{-1}  &\udots &              & \vdots   \\
			x_{0}   & x_{-1} &x_{-2}  &\cdots & \cdots  &  x_{-n+1}
		\end{bmatrix}.
		\]
		For $X \in \R^{n \times n}$, denote by $\vec_H(X)$ the following vector:
		\begin{equation}\label{eq10}
			\vec_H(X)=  \left[x_{n-1},  x_{n-2}, \ldots, x_{1}, x_{0}, x_{-1}, \ldots, x_{-n+ 1}\right]^T \in \R^{2n-1}.
		\end{equation}
	\end{definition} 
	\begin{definition}
		A matrix $X \in \R^{n \times n}$ is circulant if it has the following form:
		\[
		X=  \begin{bmatrix}
			x_{0}   & x_{n-1}&\cdots& x_{2}  & x_{1} \\
			x_{1}   & x_{0}    &x_{n-1}&           & x_{2}  \\
			\vdots   & x_{1}    &x_{0} &\ddots  & \vdots  \\
			x_{n-2}   &               &\ddots  &\ddots & x_{n-1}   \\
			x_{n-1} & x_{n-2}&\cdots   & x_{1}     & x_{0}
		\end{bmatrix}.
		\]
		For $X \in \R^{n \times n}$, denote by $\vec_C(X)$ the following vector:
		\begin{equation}\label{eq11}
			\vec_C(X)= \left[x_{0}, x_{1}, x_{2}, \ldots, x_{n-1}\right]^T \in \R^{n}.
		\end{equation}
	\end{definition} 
In the following four lemmas, we  describe the structure of some particular classes of real matrix sets. 
	\begin{lemma} \label{lem3.9}
		If $X \in \R^{n \times n}$, then $X \in \mathbb{T}\R^{n \times n} \iff \vec(X)=  K_T \vec_T(X),$
		where $\vec_T(X)$ is of the form \eqref{eq9}, and the matrix $K_T \in \R^{n^2 \times 2n-1}$ is represented as
		\begin{equation*}
			K_T= \begin{bmatrix}
				e_n& e_{n-1}& e_{n-2} & \cdots & e_2    & e_1      & 0         & \cdots& 0       & 0 \\
				0    & e_n      & e_{n-1}&\cdots &  e_3         & e_2     & e_1      & \cdots& 0        & 0 \\
				\vdots   & \vdots  & \vdots &           &\vdots       &  \vdots &            & & \vdots&\vdots\\
				0   &0            &0          &  \cdots & e_n   &e_{n-1} & \cdots&e_2 &e_1     &0 \\
				0   &0            & 0          &   \cdots & 0      & e_n     &e_{n-1}   &\cdots &e_2     &e_1
			\end{bmatrix}.
		\end{equation*} 
	\end{lemma}
	\begin{proof}
	Consider $ X=  \begin{bmatrix}
		x_{0}      & x_{1}   & x_{2}  &\cdots & \cdots  & x_{n-1} \\
		x_{-1}     & x_{0}  & x_{1}   &\ddots &             & \vdots  \\
		x_{-2}    & x_{-1} &\ddots  &\ddots &\ddots  & \vdots  \\
		\vdots    & \ddots&\ddots  &\ddots & x_{1}    & x_{2}   \\
		\vdots    &            &\ddots  &x_{-1}  & x_{0}   & x_{1}   \\
		x_{-n+1}& \cdots&\cdots  &x_{-2}  & x_{-1}  & x_{0}
	\end{bmatrix}.$ Clearly, $X$ is a Toeplitz matrix.\\
	Let $u_{i}$, for $i=1,2,\ldots,n$, denote the $i^{th}$ column of matrix $X$. We have
	$\vec(X)= \begin{bmatrix}
		u_1 \\
		u_2\\
		\vdots \\
		u_n
	\end{bmatrix}.$\
	We get
	\begin{equation*}
		u_1 = \begin{bmatrix}
			0 & 0 &  0& \cdots  & 0 & 1 & 0 &  \cdots & 0 & 0 \\
			0 & 0 & 0 &  \cdots  & 1 & 0 & 0 & \cdots & 0 & 0 \\
			\vdots &\vdots &\vdots & \udots& \vdots & \vdots & \vdots &  & \vdots& \vdots   \\
			0& 0 & 1& \cdots& 0 & 0 & 0 & \cdots & 0& 0 \\
			0 & 1 &  0& \cdots  & 0 & 0 & 0 & \cdots & 0 & 0 \\
			1 & 0 &  0& \cdots  & 0 & 0 & 0 & \cdots & 0 & 0 
		\end{bmatrix} \vec_{T}(X)  
		= \begin{bmatrix}
			e_{n} & e_{n-1} &  e_{n-2} & \cdots  & e_2 & e_1 & 0 & \cdots & 0 & 0 
		\end{bmatrix} \vec_{T}(X) ,
	\end{equation*}
	\begin{equation*}
		u_2 = \begin{bmatrix}
			0 & 0 &  0 &\cdots   & 0 & 0 & 1 & \cdots & 0 & 0 \\
			0 & 0 &  0 &\cdots  & 0 & 1 & 0 & \cdots & 0 & 0 \\
				0 & 0 &  0 &\cdots&1&0&0& \cdots & 0 & 0 \\
			\vdots &\vdots &\vdots & \udots & \vdots&\vdots & \vdots &  & \vdots& \vdots   \\
			0 & 0& 1 &  \cdots  & 0 & 0 & 0 & \cdots & 0 & 0 \\
			0 & 1 &  0&\cdots  & 0 & 0 & 0 & \cdots & 0 & 0 
		\end{bmatrix} \vec_{T}(X) 
		= \begin{bmatrix}
			0 & e_{n} & e_{n-1} &  \cdots  & e_3 & e_2 & e_1 & \cdots & 0 & 0 
		\end{bmatrix} \vec_{T}(X) ,
	\end{equation*}
	\begin{equation*}
		u_n = \begin{bmatrix}
			0 & 0 &  0 &\cdots  & 0 & 0 & 0 &  \cdots & 0 & 1 \\
			0 & 0 &  0 &\cdots  & 0 & 0 & 0 &  \cdots & 1 & 0 \\
			\vdots &\vdots & \vdots& & \vdots & \vdots&\vdots &  & \vdots& \vdots   \\
			0 & 0 &  0 &\cdots  & 0 & 0& 1& \cdots & 0 & 0 \\
			0 & 0 &  0 &\cdots  & 0& 1 & 0 & \cdots & 0 & 0 
		\end{bmatrix} \vec_{T}(X)  
		= \begin{bmatrix}
			0 & 0 &  0 & \cdots  & 0 & e_n & e_{n-1} & \cdots & e_2 & e_1 
		\end{bmatrix} \vec_{T}(X) .
	\end{equation*}
	We have
	\begin{equation*}
		\vec(X) = \begin{bmatrix}
			e_n& e_{n-1}& e_{n-2} & \cdots & e_2    & e_1      & 0         & \cdots& 0       & 0 \\
			0    & e_n      & e_{n-1}&\cdots &  e_3         & e_2     & e_1      & \cdots& 0        & 0 \\
			\vdots   & \vdots  & \vdots &           &\vdots       &  \vdots &            & & \vdots&\vdots\\
			0   &0            &0          &  \cdots & e_n   &e_{n-1} & \cdots&e_2 &e_1     &0 \\
			0   &0            & 0          &   \cdots & 0      & e_n     &e_{n-1}   &\cdots &e_2     &e_1
		\end{bmatrix} \vec_{T}(X) = K_{T} \vec_{T}(X). 
	\end{equation*}
\end{proof}
	\begin{lemma} \label{lem3.91}
	If $X \in \R^{n \times n}$, then $X \in \mathbb{S}\mathbb{T}\R^{n \times n} \iff \vec(X)=  K_{ST} \vec_{ST}(X),$
	where $\vec_{ST}(X)$ is of the form \eqref{eq3.1}. When $n$ is even, let $n=2l$. In this case, the matrix $K_{ST} \in \R^{n^2 \times n}$ is represented as
	\begin{equation*}
		K_{ST}= \begin{bmatrix}
			e_1& e_{2}& e_{3} & \cdots    & e_{l}      & e_{l+1}         & \cdots& e_{n-1}       & e_{n} \\
			e_{2}    & e_1+e_3      & e_{4}&\cdots        & e_{l+1}     & e_{l+2}      & \cdots& e_{n}        & 0 \\
			e_{3}    & e_2+e_4      & e_{1}+e_5&\cdots         & e_{l+2}     & e_{l+3}      & \cdots& 0        & 0 \\
			\vdots   & \vdots  & \vdots &      \ddots            &  \vdots &  \vdots           && \vdots&\vdots\\
				e_{l}    & e_{l-1}+e_{l+1}      & e_{l-2}+e_{l+2}&\cdots         & e_1+e_{n-1}     & e_{n}      & \cdots& 0        & 0 \\
					e_{l+1}    & e_{l}+e_{l+2}      & e_{l-1}+e_{l+3}&\cdots         & e_2+e_{n}     & e_{1}      & \cdots& 0        & 0 \\
					\vdots   & \vdots  & \vdots &                 &  &            &\ddots & \vdots&\vdots\\
			e_{n-1}   &e_{n-2}+e_{n}            & e_{n-3}         &   \cdots      & \cdots     &\cdots &\cdots &e_1    &0\\
			e_{n}   &e_{n-1}            & e_{n-2}          &   \cdots      & \cdots     &\cdots   &\cdots &e_2     &e_1
		\end{bmatrix}.
	\end{equation*} 
When $n$ is odd, let $n=2l-1$. In this case, the matrix $K_{ST} \in \R^{n^2 \times n}$ is represented as
	\begin{equation*}
	K_{ST}= \begin{bmatrix}
		e_1& e_{2}& e_{3} & \cdots    & e_{l}      & e_{l+1}         & \cdots& e_{n-1}       & e_{n} \\
		e_{2}    & e_1+e_3      & e_{4}&\cdots        & e_{l+1}     & e_{l+2}      & \cdots& e_{n}        & 0 \\
		e_{3}    & e_2+e_4      & e_{1}+e_5&\cdots         & e_{l+2}     & e_{l+3}      & \cdots& 0        & 0 \\
		\vdots   & \vdots  & \vdots &      \ddots            &  \vdots &  \vdots          & & \vdots&\vdots\\
		e_{l}    & e_{l-1}+e_{l+1}      & e_{l-2}+e_{l+2}&\cdots         & e_1+e_{n}     & 0      & \cdots& 0        & 0 \\
		e_{l+1}    & e_{l}+e_{l+2}      & e_{l-1}+e_{l+3}&\cdots         & e_2               & e_{1}      & \cdots& 0        & 0 \\
		\vdots   & \vdots  & \vdots &                 &   &           &\ddots & \vdots&\vdots\\
		e_{n-1}   &e_{n-2}+e_{n}            & e_{n-3}         &   \cdots      & \cdots     &\cdots  &\cdots &e_1    &0\\
		e_{n}   &e_{n-1}            & e_{n-2}          &   \cdots      & \cdots     &\cdots   &\cdots &e_2     &e_1
	\end{bmatrix}.
\end{equation*} 
\end{lemma}
	\begin{proof}
	The proof is similar to Lemma \ref{lem3.9}. 
\end{proof}
To enhance our understanding of the above lemma, we will examine it in the context of $n = 6$ and $n=7$. In this scenario, we have
\begin{equation*}
	K_{ST}=\begin{bmatrix}
		e_1      & e_2         & e_3             & e_4              & e_5                 & e_6           \\
		e_2     &  e_1+e_3&e_4              & e_5               & e_6                 & 0    \\
		e_3     & e_2+e_4 & e_1+e_5    & e_6               & 0                      & 0 \\
		e_4     & e_3+e_5 & e_2+e_6   & e_1                & 0                      & 0  \\
		e_5     & e_4+e_6  & e_3           & e_2                & e_1                  & 0   \\
		e_6      & e_5         & e_4             & e_3              & e_2                 & e_1 
	\end{bmatrix}, \;
K_{ST}=\begin{bmatrix}
	e_1      & e_2         & e_3             & e_4              & e_5                 & e_6         &   e_7  \\
	e_2     &  e_1+e_3&e_4              & e_5               & e_6                 & e_7         &  0     \\
	e_3     & e_2+e_4 & e_1+e_5    & e_6               & e_7                 &0               & 0 \\
	e_4     & e_3+e_5 & e_2+e_6   & e_1 +e_7       & 0                     & 0              &0   \\
	e_5     & e_4+e_6  & e_3+e_7   & e_2                & e_1                  & 0             &0    \\
	e_6      & e_5+e_7 & e_4             & e_3              & e_2                 & e_1           &0    \\
	e_7      & e_6         & e_5             & e_4              & e_3                 & e_2         &   e_1
\end{bmatrix}.
\end{equation*}
	\begin{lemma} \label{lem3.10}
		If $X \in \R^{n \times n}$, then $X \in \mathbb{H}\R^{n \times n} \iff \vec(X)= K_H \vec_H(X),$
		where $\vec_H(X)$ is of the form \eqref{eq10}, and the matrix $K_H \in \R^{n^2 \times 2n-1}$ is represented as
		\begin{equation*}
			K_H= \begin{bmatrix}
				e_1  &e_2   &  e_3  & \cdots&e_{n-1}&e_n    &0        &0          &\cdots & 0 \\                               
				0     &e_1    &  e_2  & \cdots&e_{n-2}&e_{n-1}&e_n        &\cdots  & 0         & 0 \\
				\vdots&\vdots&\vdots&            &\vdots  &\vdots&&            &\vdots  &\vdots \\
				0   &0         &0        & \cdots     &e_1      & e_2         &\cdots&e_{n-1}&e_n      &0   \\
				0   &0         &0        &\cdots       &0         &e_1     &e_2    &\cdots &e_{n-1}&e_n
			\end{bmatrix}.
		\end{equation*} 
	\end{lemma}
	\begin{proof}
	The proof is similar to Lemma \ref{lem3.9}. 
	\end{proof}
	\begin{lemma} \label{lem3.11}
		If $X \in \R^{n \times n}$, then $X  \in \mathbb{C}\R^{n \times n} \iff \vec(X)= K_C \vec_C(X), $
		where $\vec_C(X)$ is of the form \eqref{eq11}, and the matrix $K_C \in \R^{n^2 \times n}$ is represented as
		\begin{equation*}
			K_C= \begin{bmatrix}
				e_1   &  e_2    &  \cdots  & e_{n-1}  &   e_n  \\
				e_2   &  e_3    &  \cdots  & e_{n}    &   e_1  \\                                
				e_3   &  e_4    &  \cdots  & e_{1}    &   e_2  \\
				\vdots&  \vdots &          & \vdots   & \vdots  \\
				e_n   &  e_1    &  \cdots  & e_{n-2}  & e_{n-1}  \\
			\end{bmatrix}.
		\end{equation*} 
	\end{lemma}
	\begin{proof}
	The proof is similar to Lemma \ref{lem3.9}. 
	\end{proof}
In the following lemma, we present the vec-structure of reduced biquaternion L-structure matrices based on the vec-structure of real structure matrices. 
	\begin{lemma}\label{lem33}
		If $X= X_1+ X_2\j \in \QR^{n \times n}$, then
		\begin{fleqn}[\parindent]
			\begin{equation*}
				(\Rn{1}) \quad  X \in \mathbb{T}\QR^{n \times n} \iff \vec(\ov{X})= M_T\vec_T(\ov{X}),    \mbox{where}       
			\end{equation*} 
		\end{fleqn}  
		\begin{equation*}
			M_T= \begin{bmatrix}
				K_T  & 0  & 0  & 0  \\
				0    & K_T& 0  & 0   \\
				0    & 0  & K_T& 0   \\
				0    & 0  & 0  & K_T
			\end{bmatrix}, \;
			\vec_T(\ov{X}) =
			\begin{bmatrix}
				\vec_T(\Re(X_1))\\
				\vec_T(\Im(X_1))\\
				\vec_T(\Re(X_2))\\
				\vec_T(\Im(X_2))
			\end{bmatrix}.
		\end{equation*}
		\begin{fleqn}[\parindent]
		\begin{equation*}
			(\Rn{2}) \quad  X \in \mathbb{S}\mathbb{T}\QR^{n \times n} \iff \vec(\ov{X})= M_{ST}\vec_{ST}(\ov{X}),  \mbox{where}         
		\end{equation*}
	\end{fleqn}
	\begin{equation*}
		M_{ST}=  \begin{bmatrix}
			K_{ST}  & 0  & 0  & 0  \\
			0    & K_{ST}& 0  & 0   \\
			0    & 0  & K_{ST}& 0   \\
			0    & 0  & 0  & K_{ST}
		\end{bmatrix}, \;
		\vec_{ST}(\ov{X})= 
		\begin{bmatrix}
			\vec_{ST}(\Re(X_1))\\
			\vec_{ST}(\Im(X_1))\\
			\vec_{ST}(\Re(X_2))\\
			\vec_{ST}(\Im(X_2))
		\end{bmatrix}.
	\end{equation*}
		\begin{fleqn}[\parindent]
			\begin{equation*}
				(\Rn{2}) \quad X \in \mathbb{H}\QR^{n \times n} \iff \vec(\ov{X})= M_H\vec_H(\ov{X}),  \mbox{where}  
			\end{equation*}	
		\end{fleqn}
		\begin{equation*}
			M_H= \begin{bmatrix}
				K_H  & 0  & 0  & 0  \\
				0    & K_H& 0  & 0   \\
				0    & 0  & K_H& 0   \\
				0    & 0  & 0  & K_H
			\end{bmatrix}, \;
			\vec_H(\ov{X}) =
			\begin{bmatrix}
				\vec_H(\Re(X_1))\\
				\vec_H(\Im(X_1))\\
				\vec_H(\Re(X_2))\\
				\vec_H(\Im(X_2))
			\end{bmatrix}.
		\end{equation*}
		\begin{fleqn}[\parindent]
			\begin{equation*}
				(\Rn{3}) \quad  X \in \C\QR^{n \times n} \iff \vec(\ov{X})= M_C\vec_C(\ov{X}),  \mbox{where}         
			\end{equation*}
		\end{fleqn}
		\begin{equation*}
			M_C=  \begin{bmatrix}
				K_C  & 0  & 0  & 0  \\
				0    & K_C& 0  & 0   \\
				0    & 0  & K_C& 0   \\
				0    & 0  & 0  & K_C
			\end{bmatrix}, \;
			\vec_C(\ov{X})= 
			\begin{bmatrix}
				\vec_C(\Re(X_1))\\
				\vec_C(\Im(X_1))\\
				\vec_C(\Re(X_2))\\
				\vec_C(\Im(X_2))
			\end{bmatrix}.
		\end{equation*}
	
	\end{lemma}
	\begin{proof}
		We will prove the first part, and the remaining parts can be proved in a similar manner.
	The proof follows from the fact that $X \in \mathbb{T}\QR^{n \times n} \iff \Re(X_1), \Im(X_1), \Re(X_2), \Im(X_2) \in \mathbb{T}\R^{n \times n}$ and using Lemma \ref{lem3.9}. 
	\end{proof}
Until now, we have examined the representation of a reduced biquaternion L-structure matrix using a real structure matrix for a particular class of matrix sets. Based on the preceding discussion about reduced biquaternion L-structure matrix sets, we can summarize the findings as follows: 

For $X= X_1+ X_2\j \in \QR^{m \times n}$, we have $\ov{X}= \left[\Re(X_1), \Im(X_1), \Re(X_2), \Im(X_2)\right] \in \R^{m \times 4n}$. Let $\mathcal{G}$ be a subspace of $\R^{4mn}$ and $M_L$ be the basis matrix for $\mathcal{G}$. The subset of real matrices of order $m \times 4n$ given by
		\begin{equation} \label{eq12}
			L^R(m,4n)= \{\ov{X} \in \R^{m \times 4n} \; | \; \vec(\ov{X}) \in \mathcal{G}\}
		\end{equation}
will be called a real L-structure.

	\begin{remark}
		 $M_L$ represents the basis matrix of the subspace $\mathcal{G}$. To simplify things, we will refer to $M_L$ as the basis matrix of $L^R(m,4n)$ throughout the entire manuscript. 
	\end{remark}
	Thus, we have the following Lemma.
	\begin{lemma} \label{lem3.16}
		Let $M_L$ be the basis matrix of $L^R(m,4n)$. Then $X \in L(m,n) \iff \vec(\ov{X})= M_L\vec_L(\ov{X}),$ where $\vec_L(\ov{X})$ corresponds to the representation of $\ov{X}$ according to the basis matrix $M_L$.
	\end{lemma} 
	\begin{proof}
	The proof follows from the generalization of Lemmas \ref{lem3.9} and \ref{lem33} to any L-structure matrix $X.$
	\end{proof}
	
	Now that we have described the reduced biquaternion L-structure, we turn our attention to solving a RBME. Our approach for addressing the RBME involves transforming it into a complex matrix equation. To achieve this, we must study $\vec(AXB)$. For $A \in \mathbb{C}^{m \times n}, X \in \mathbb{C}^{n \times s}, \; \mbox{and} \; B \in \mathbb{C}^{s \times t}$, we have
	\begin{equation} \label{eq13}
		\vec(AXB)= (B^T \otimes A)\vec(X).
	\end{equation}
	However, this differs in the context of reduced biquaternion algebra. Thus, we investigate  $\vec(\Psi_{AXB})$ rather than $\vec(AXB)$ in reduced biquaternion algebra.
	\begin{lemma} \label{lem3.17}
		Let $A= A_1+ A_2\j \in \QR^{m \times n}$, $X= X_1+ X_2\j \in \QR^{n \times s}, \; \mbox{and} \; B= B_1+ B_2\j \in \QR^{s \times t}$. Then
		\begin{equation*}
			\vec(\Psi_{AXB}) = \left(h(B)^{T} \otimes A_1 +  h(B\j)^{T} \otimes A_2\right)\vec(\Psi_{X}).
		\end{equation*}
	\end{lemma}
	\begin{proof}
	Using \eqref{eq4} and \eqref{eq5}, we get
			$\Psi_{AXB}  =  \Psi_{A}h(XB) 
			= \left[A_1, A_2\right]h(X)h(B)
			= \left[A_1, A_2\right]
			\begin{bmatrix}
				X_1 & X_2 \\
				X_2 & X_1
			\end{bmatrix}
			\begin{bmatrix}
				B_1 & B_2 \\
				B_2 & B_1
			\end{bmatrix} 
			= \left[
			A_1X_1B_1+ A_2X_2B_1+ A_1X_2B_2+ A_2X_1B_2,  A_1X_1B_2+ A_2X_2B_2+ A_1X_2B_1+ A_2X_1B_1
			\right]. $\\
			
	Now from  \eqref{eq6} and \eqref{eq13}, we get
	\begin{equation*}
		\begin{split}
			\vec(\Psi_{AXB}) & =  \begin{bmatrix}
				(B_1^T \otimes A_1)\vec(X_1)+(B_1^T \otimes A_2)\vec(X_2) + (B_2^T \otimes A_1)\vec(X_2) +(B_2^T \otimes A_2)\vec(X_1) \\
				(B_2^T \otimes A_1)\vec(X_1) + (B_2^T \otimes A_2)\vec(X_2) + (B_1^T \otimes A_1)\vec(X_2) + (B_1^T \otimes A_2)\vec(X_1)
			\end{bmatrix} \\
			& = \left(
			\begin{bmatrix}
				B_1^T & B_2^T \\
				B_2^T & B_1^T
			\end{bmatrix} \otimes A_1 + 
			\begin{bmatrix}
				B_2^T & B_1^T \\
				B_1^T & B_2^T
			\end{bmatrix} \otimes A_2
			\right)
			\begin{bmatrix}
				\vec(X_1) \\
				\vec(X_2)
			\end{bmatrix}\\
		& =  \left(
			h(B)^T \otimes A_1 + h(B\j)^T \otimes A_2
			\right) \vec(\Psi_X).  
		\end{split}
	\end{equation*}
	\end{proof}
	Set 
	\begin{equation}\label{eq14}
		\mathcal{W}_{ns}= \begin{bmatrix}
			I_{ns}  &  \i \, I_{ns}  &  0       &       0  \\
			0       &   0            &  I_{ns}  & \i \, I_{ns}
		\end{bmatrix},\; \;
	\mathcal{S}_{ns} =  \begin{bmatrix}
Q_{ns}   &     0   \\
0        &    Q_{ns}
\end{bmatrix},    
	\end{equation}
where $Q_{ns}$ is the commutation matrix, a row permutation of the identity matrix $I_{ns}$.

We have examined $\vec(\Psi_{AXB})$ within reduced biquaternion algebra. The following lemma outlines $\vec(\Psi_{AXB})$ when $X$ possesses an L-structure in reduced biquaternion algebra.
	\begin{lemma} \label{lem3.18}
		Let $A= A_1+ A_2\j \in \QR^{m \times n}, X= X_1+ X_2\j \in L(n,s), \; \mbox{and} \; B= B_1+ B_2\j \in \QR^{s \times t}$. Then
		\begin{eqnarray*}
			\vec(\Psi_{AXB}) &=& \left(h(B)^T \otimes A_1+h(B\j)^T \otimes A_2\right) \mathcal{W}_{ns} M_L \vec_L(\ov{X}), \\
				\vec(\Psi_{AX^TB}) &=& \left(h(B)^T \otimes A_1+ h(B\j)^T \otimes A_2\right) \mathcal{S}_{ns}\mathcal{W}_{ns}  M_L \vec_L(\ov{X}).
			\end{eqnarray*}
	\end{lemma}
	\begin{proof}
	Using  \eqref{eq6}, \eqref{eq14}, Lemma \ref{lem3.16}, and Lemma \ref{lem3.17}, we get
	\begin{equation*}
		\begin{split}
			\vec(\Psi_{AXB}) & = \left(h(B)^{T} \otimes A_1 +  h(B\j)^{T} \otimes A_2\right)\vec(\Psi_{X})\\ &= \left(h(B)^{T} \otimes A_1 +  h(B\j)^{T} \otimes A_2\right)
			\begin{bmatrix}
				\vec(X_1) \\
				\vec(X_2)
			\end{bmatrix} \\
			& =  \left(h(B)^{T} \otimes A_1 +  h(B\j)^{T} \otimes A_2\right)
			\begin{bmatrix}
				\vec(\Re(X_1))+\i \, \vec(\Im(X_1)) \\
				\vec(\Re(X_2))+\i \, \vec(\Im(X_2)) 
			\end{bmatrix} \\
			& =  \left(h(B)^{T} \otimes A_1 +  h(B\j)^{T} \otimes A_2\right) \mathcal{W}_{ns}  \vec(\ov{X}) \\ &=  \left(h(B)^{T} \otimes A_1 +  h(B\j)^{T} \otimes A_2\right) \mathcal{W}_{ns} \mathcal{M}_{L} \vec_L(\ov{X}).
		\end{split}
	\end{equation*}
We have
\begin{equation*}
	\vec(\Psi_{X^T}) = 
	\begin{bmatrix}
		\vec(X_1^T) \\
		\vec(X_2^T)
	\end{bmatrix} =  
	\begin{bmatrix}
		\mathcal{Q}_{ns}\vec(X_1) \\
		\mathcal{Q}_{ns}\vec(X_2)
	\end{bmatrix} = \mathcal{S}_{ns}
	\begin{bmatrix}
		\vec(X_1) \\
		\vec(X_2)
	\end{bmatrix} =\mathcal{S}_{ns}\mathcal{W}_{ns} \mathcal{M}_{L} \vec_L(\ov{X}) .
\end{equation*}
The proof follows from some simple calculations.
\end{proof}
	\section{General framework for solving constrained RBMEs}  \label{sec4}
	The purpose of this section is to demonstrate how we can solve constrained generalized linear matrix equations over commutative quaternions. As part of our approach, the constrained RBME is reduced to the following unconstrained real matrix system:
	\begin{equation}\label{eq15}
		\begin{bmatrix}
			Q_1   \\
			Q_2
		\end{bmatrix}x= e,
	\end{equation}where $Q_1, Q_2$ are real matrices of appropriate dimension and $x, e$ are real vectors of suitable size. From \cite[Theorem 2]{MR172890} the generalized inverse of a partitioned matrix $\left[U,  V\right]$ is given by
	$$
	\left[U,  V\right]^{+} = 
	\begin{bmatrix}
		U^{+} -  U^{+}VH \\
		H
	\end{bmatrix},
	$$
	where 
	\begin{equation*}
		\left.\begin{aligned}
			 	H &=  R^+ + \left(I-R^+R\right)ZV^TU^{+T}U^+\left(I- VR^+\right),    \; R  =  \left(I- UU^+\right)V,\\
			Z    & = \left(I+ \left(I- R^+R\right)V^TU^{+T}U^+V\left(I- R^+R\right)\right)^{-1}. 
		\end{aligned}\right\}
	\end{equation*}
	We have
	$$\left[U,  V\right]^{T+}= \left[U,  V\right]^{+T} = 
	\begin{bmatrix}
		U^{+} -  U^{+}VH \\
		H
	\end{bmatrix}^T= \left[U^{T+}-H^{T}V^{T}U^{T+},  H^T\right].$$
	By substituting $U=Q_1^T$ and $V=Q_2^T$, we get
	\begin{equation}\label{eq16}
		\begin{bmatrix}
			Q_1   \\
			Q_2
		\end{bmatrix}^+= \left[Q_1^+ - H^T Q_2Q_1^+,  H^T\right], \;
		\begin{bmatrix}
			Q_1 \\
			Q_2
		\end{bmatrix}^+
		\begin{bmatrix}
			Q_1 \\
			Q_2
		\end{bmatrix}= Q_1^+Q_1+ R R^+,
	\end{equation}
	where
	\begin{equation}\label{eq17}
		\left.\begin{aligned}
			H &=  R^+ + \left(I-R^+R\right)ZQ_2Q_1^+Q_1^{+T}\left(I- Q_2^TR^+\right), \; R =  \left(I- Q_1^+ Q_1\right)Q_2^T,  \;	\\
			Z    & = \left(I+ \left(I- R^+R\right)Q_2Q_1^+Q_1^{+T}Q_2^T\left(I- R^+R\right)\right)^{-1}. 
		\end{aligned}\right\}
	\end{equation}
	Using \cite{golub2013matrix} and the results mentioned above, we deduce the following lemma that is helpful in developing the main results.
	\begin{lemma}  \label{lem4.1}   Consider the real matrix system of the form $	\begin{bmatrix}
			Q_1   \\
			Q_2
		\end{bmatrix}x= e$. We have the following results:
		\begin{enumerate}
			\item[$(1)$] The matrix equation has a solution $x$ if and only if 
			$\begin{bmatrix}
				Q_1   \\
				Q_2
			\end{bmatrix}
			\begin{bmatrix}
				Q_1   \\
				Q_2
			\end{bmatrix}^+e = e$. In this case, the general solution is $x= \left[Q_1^+ - H^T Q_2Q_1^+,  H^T\right]e+\left(I-Q_1^+Q_1-RR^+\right)y,$ where $y$ is an arbitrary vector of suitable size. Furthermore, if the consistency condition is satisfied, then the matrix equation has a unique solution if and only if matrix $\begin{bmatrix}
			Q_1   \\
			Q_2
		\end{bmatrix}$ is of full column rank. In this case, the unique solution is $x=  \left[Q_1^+ - H^T Q_2Q_1^+,  H^T\right]e$.
			\item [$(2)$] The least squares solutions of the matrix equation can be expressed as
			$x= \left[Q_1^+ - H^T Q_2Q_1^+,  H^T\right]e+\left(I-Q_1^+Q_1-RR^+\right)y,$ where $y$ is an arbitrary vector of suitable size, and the least squares solution with the least norm is $x=  \left[Q_1^+ - H^T Q_2Q_1^+,  H^T\right]e$.
		\end{enumerate}
	\end{lemma}
	The following lemma will be used for the development of main results.
	\begin{lemma}\label{lem4.2}
		Consider the matrix equation $Ax=b,$ where $A \in \C^{m \times n}, x \in \R^{n}$, and $b \in \C^{m}.$ The matrix equation $Ax=b$ is equivalent to the linear system $\begin{bmatrix}
			\Re(A) \\
			\Im(A)
		\end{bmatrix}x=
		\begin{bmatrix}
			\Re(b)\\
			\Im(b)
		\end{bmatrix}.$
	\end{lemma}
 In the following subsection, we aim to find $Q_1$, $Q_2$, and $e$ for each of the three constrained RBMEs and solve them. 
	\begin{remark}
			It is important to emphasize that the values of $Q_1$, $Q_2$, and $e$ vary depending on the specific matrix equation we are attempting to solve.
	\end{remark}
\subsection{Linear matrix equation in several unknown L-structures} \label{sec4.2}
		The class of matrix equation \eqref{eq2} encompasses many important matrix equations. Some simple examples are $AXB+ CYD= E, \; AX+ YB= E$. We now introduce a general framework for finding the least squares solutions of RBME of the form \eqref{eq2}. The problem can be formally stated as follows:
	\begin{problem} \label{pb4.1}
		Let $A_l= A_{l1}+ A_{l2}\j \in \QR^{m \times n_l}$, $B_l \in \QR^{s_l \times t}$, and $E= E_1+ E_2\j \in \QR^{m \times t}$ for $l= 1, 2, \ldots, r$. Let
		\begin{equation*}
			\mathcal{N}_{LE}= \left\{[X_1, X_2, \ldots, X_r] \; | \; X_{l} \in L_l(n_l,s_l), \;  \norm{\sum_{l= 1}^{r}A_lX_lB_l- E}_F= \min_{\substack{\widetilde{X}_{l} \in L_l(n_l, s_l) }} \norm{\sum_{l= 1}^{r}A_l\widetilde{X}_{l}B_l- E}_F\right\}.
		\end{equation*}
		Then find $[X_{1E}, X_{2E}, \ldots, X_{rE}] \in \mathcal{N}_{LE}$ such that
		\begin{equation*}
			\norm{[X_{1E}, X_{2E}, \ldots, X_{rE}]}_F= \min_{[X_1, X_2, \ldots, X_r] \in \mathcal{N}_{LE}}\left(\norm{X_1}_F^2+ \norm{X_2}_F^2+ \cdots+ \norm{X_r}_F^2\right)^{\frac{1}{2}}.
		\end{equation*}
	\end{problem}
	\noindent
To solve Problem \ref{pb4.1}, we employ the following notations: for $l= 1, 2, \ldots, r$, let $M_{L_l}$ be the basis matrix of $L_l^R(n_l,4s_l)$, and 
	\begin{eqnarray}
		S_l &:=& \left(h(B_l)^T \otimes A_{l1}+ h(B_{l}\j)^T \otimes A_{l2}\right) \mathcal{W}_{n_{l}s_{l}} M_{L_{l}}, \label{eq4.25}\\
		x &:=& \begin{bmatrix}
			\vec_{L_1}(\ov{X_1}) \\
			\vec_{L_2}(\ov{X_2}) \\
			\vdots \\
			\vec_{L_r}(\ov{X_r}) \\
		\end{bmatrix}.  \label{eq4.26}
	\end{eqnarray}
	Additionally, $Q_1, Q_2$, and $e$ (as in \eqref{eq15}) are in the following form:
	\begin{equation}\label{eq25}
		Q_1:= \left[
			\Re(S_1), \Re(S_2), \ldots  , \Re(S_r)  
		\right], \;
		Q_2:= \left[
			\Im(S_1), \Im(S_2), \ldots, \Im(S_r)
		\right], \; \mbox{and} \;
		e:= \begin{bmatrix}
			\vec(\Re(\Psi_{E}))  \\
			\vec(\Im(\Psi_{E}))
		\end{bmatrix}.
	\end{equation}
	In case of inconsistency in matrix equation \eqref{eq2}, we  provide the least squares solutions. The following result provides the solution to Problem \ref{pb4.1}.
	\begin{theorem}\label{thm4.1}
		Let $A_l \in \QR^{m \times n_l}$, $B_l \in \QR^{s_l \times t}$, and $E \in \QR^{m \times t}$ for $l= 1, 2, \ldots, r$. Let $Q_1, Q_2$, and $e$ be of the form \eqref{eq25} and $\mathcal{T}=\diag(	M_{L_1}, M_{L_2} , \ldots,  M_{L_r} )$. Then
		\begin{equation}\label{eq4.1}
			\mathcal{N}_{LE}= \left\{[X_1, X_2,\ldots, X_r] \; 
			\left| \; \begin{bmatrix}
				\vec(\ov{X_1})  \\
				\vec(\ov{X_2})   \\
				\vdots  \\
				\vec(\ov{X_r})
			\end{bmatrix} \right.= 
			\mathcal{T}\left[Q_1^+- H^TQ_2Q_1^+, H^T\right]e+	\mathcal{T} \left(I- Q_1^+Q_1- R R^+\right)y\right\},
		\end{equation}
		where $y$ is any vector of suitable size. The unique solution $[X_{1E}, X_{2E}, \ldots, X_{rE}] \in \mathcal{N}_{LE}$ to Problem \ref{pb4.1} satisfies
		\begin{equation}\label{eq4.2}
			\begin{bmatrix}
				\vec(\ov{X_{1E}})  \\
				\vec(\ov{X_{2E}})   \\
				\vdots  \\
				\vec(\ov{X_{rE}})
			\end{bmatrix}= 
		\mathcal{T}\left[Q_1^+- H^TQ_2Q_1^+, H^T\right]e.
		\end{equation}
	\end{theorem}
	\begin{proof}
	By using \eqref{eq7}, we get
	\begin{eqnarray*}	
		\norm{\sum_{l= 1}^{r}A_lX_lB_l- E}_F^2  &=&  	\norm{\sum_{l= 1}^{r} \Psi_{A_lX_lB_l} - \Psi_E}_F^2 \\
		&=& \norm{\sum_{l= 1}^{r} \vec\left(\Psi_{A_lX_lB_l}\right)- \vec\left(\Psi_E\right)}_F^2 .
	\end{eqnarray*}
	Using Lemma \ref{lem3.18}, we have
	\begin{equation*}
		\vec\left(\Psi_{A_lX_lB_l}\right)  =  \left(h(B_l)^T \otimes A_{l1}+ h(B_l\j)^T \otimes A_{l2}\right)\mathcal{W}_{n_ls_l}M_{L_l}\vec_{L_l}(\ov{X_l}).
	\end{equation*}
	Now using \eqref{eq4.25}, we have
	\begin{eqnarray*}
		\sum_{l= 1}^{r} \vec\left(\Psi_{A_lX_lB_l}\right)  &=&  \sum_{l= 1}^{r}\left(h(B_l)^T \otimes A_{l1}+ h(B_l\j)^T \otimes A_{l2}\right)\mathcal{W}_{n_ls_l}M_{L_l}\vec_{L_l}(\ov{X_l}) \\
		& = &  \sum_{l= 1}^{r} S_l \vec_{L_l}(\ov{X_l}).
	\end{eqnarray*}
By using \eqref{eq4.26}, \eqref{eq25}, and Lemma \ref{lem4.2}, we get
	\begin{eqnarray*}
		\norm{\sum_{l= 1}^{r}A_lX_lB_l- E}_F^2 	& =  & \norm{	\sum_{l= 1}^{r}S_l \vec_{L_l}(\ov{X_l})- \vec\left(\Psi_E\right)}_F^2 \\
		&=& \norm{\left[
				S_1, S_2, \ldots, S_r
			\right]\begin{bmatrix}
				\vec_{L_1}(\ov{X_1}) \\
				\vec_{L_2}(\ov{X_2}) \\
				\vdots \\
				\vec_{L_r}(\ov{X_r}) \\
			\end{bmatrix}- \vec\left(\Psi_E\right)}_F^2 \\
		& =  & \norm{\begin{bmatrix}
				\Re(S_1)  & \Re(S_2) & \cdots  & \Re(S_r)  \\
				\Im(S_1)   &   \Im(S_2)  & \cdots  & \Im(S_r)
			\end{bmatrix}x- 
			\begin{bmatrix}
				\vec\left(\Re(\Psi_E)\right) \\
				\vec\left(\Im(\Psi_E)\right)
		\end{bmatrix}}_F^2 \\
		& =  & \norm{\begin{bmatrix}
				Q_1 \\
				Q_2
			\end{bmatrix}x- e}_F^2.
	\end{eqnarray*}
	Hence, Problem \ref{pb4.1} can be solved by finding the least squares solutions of the following unconstrained real matrix system: 
	\begin{equation*}
	\begin{bmatrix}
		Q_1 \\
		Q_2
	\end{bmatrix} x = e.
\end{equation*}
	By using Lemma \ref{lem4.1}, the least squares solutions of the above real matrix system is:
\[
x = 	\left[Q_1^+ - H^T Q_2Q_1^+, H^T\right]e+\left(I-Q_1^+Q_1-RR^+\right)y,
\]
where $y$ is any vector of suitable size, and the least squares solution with the least norm is $\left[Q_1^+ - H^T Q_2Q_1^+, H^T\right]e.$
Using Lemma \ref{lem3.16}, we have
\begin{equation*}
\begin{bmatrix}
	\vec(\ov{X_1})  \\
	\vec(\ov{X_2})   \\
	\vdots  \\
	\vec(\ov{X_r})
\end{bmatrix}  =\mathcal{T}x.
\end{equation*}
 Thus, we can obtain \eqref{eq4.1} and \eqref{eq4.2}. 
\end{proof}
The following theorem presents the consistency condition for obtaining the solution $X_l \in L_l(n_l, s_l)$ for the RBME of the form \eqref{eq2} and a general formulation for the solution. 
	
	\begin{theorem}
		Consider the RBME of the form \eqref{eq2} and let $\mathcal{T}=\diag(M_{L_1}, M_{L_2} , \ldots,  M_{L_r} )$. Then the matrix equation \eqref{eq2} has an L-structure solution $X_l \in L_l(n_l, s_l)$, for $l= 1, 2, \ldots, r$, if and only if
		\begin{equation}\label{eq26}
			\begin{bmatrix}
				Q_1  \\
				Q_2
			\end{bmatrix}
			\begin{bmatrix}
				Q_1  \\
				Q_2
			\end{bmatrix}^+e= e,
		\end{equation}	
		where $Q_1, Q_2$, and $e$ are in the form of \eqref{eq25}. In this case, the general solution $X_l \in L_l(n_l, s_l)$ satisfies
		\begin{equation*}
			\begin{bmatrix}
				\vec(\ov{X_1})  \\
				\vec(\ov{X_2})   \\
				\vdots  \\
				\vec(\ov{X_r})
			\end{bmatrix} = 
		\mathcal{T} \left[Q_1^+- H^TQ_2Q_1^+, H^T\right]e+ 	\mathcal{T}  \left(I- Q_1^+Q_1- R R^+\right)y,
		\end{equation*}
		where $y$ is any vector of suitable size. Further, if the consistency condition holds, then the RBME of the form \eqref{eq2} has a unique solution $X_l \in L_l(n_l, s_l)$ if and only if 
		\begin{equation*}
			\rank \left(\begin{bmatrix}
				Q_1 \\
				Q_2		
			\end{bmatrix}\right)
			= \dim  \left(\begin{bmatrix}
				\vec_{L_{1}}(\ov{X_1})  \\
				\vec_{L_{2}}(\ov{X_2})   \\
				\vdots  \\
				\vec_{L_{r}}(\ov{X_r})
			\end{bmatrix}\right).
		\end{equation*}
	In this case, the unique solution $X_l \in L_l(n_l, s_l)$ satisfies
		\begin{equation*}
			\begin{bmatrix}
				\vec(\ov{X_1})  \\
				\vec(\ov{X_2})   \\
				\vdots  \\
				\vec(\ov{X_r})
			\end{bmatrix}= 
		\mathcal{T} \left[Q_1^+- H^TQ_2Q_1^+, H^T\right]e.
		\end{equation*}
	\end{theorem}
	\begin{proof}
	The proof follows using Lemma \ref{lem4.1} and from the fact that
	\begin{equation*}
		\sum_{l= 1}^{r}A_lX_lB_l= E \iff \begin{bmatrix}
			\Re(S_1)   &   \Re(S_2) & \cdots   & \Re(S_r)  \\
			\Im(S_1)   &   \Im(S_2)  & \cdots & \Im(S_r)
		\end{bmatrix}
		\begin{bmatrix}
			\vec_{L_{1}}(\ov{X_1})  \\
			\vec_{L_{2}}(\ov{X_2})   \\
			\vdots  \\
			\vec_{L_{r}}(\ov{X_r})
		\end{bmatrix} =
		\begin{bmatrix}
			\vec(\Re(\Psi_E))  \\
			\vec(\Im(\Psi_E))
		\end{bmatrix}.
	\end{equation*}
\end{proof}
The remaining subsection focuses on addressing the least squares problem associated with matrix equations \eqref{eq1} and \eqref{eq3}. This involves finding the least squares solutions for the following unconstrained real matrix system:
\begin{equation}\label{eq4.28}
	\begin{bmatrix}
		Q_1 \\
		Q_2
	\end{bmatrix} \vec_{L}(\ov{X}) = e.
\end{equation}
Let $M_{L}$ be the basis matrix of $L^R(n,4s)$. Using Lemma \ref{lem3.16}, we get $\vec(\ov{X})$ from $\vec_{L}(\ov{X})$ in the following way:
\[
\vec(\ov{X}) = M_L  \vec_{L}(\ov{X}).
\]
The methodology for solving RBMEs of the form \eqref{eq1} and \eqref{eq3} remains the same as outlined in Subsection \ref{sec4.2}. Therefore, our focus here is solely on presenting the values for $Q_1$, $Q_2$, and $e$, while intentionally omitting the detailed results.\\
\textbf{Linear matrix equation in one unknown L-structure}\\
Consider the matrix equation \eqref{eq1} and let $A_l= A_{l1}+ A_{l2}\j \in \QR ^{m \times n}, B_l \in \QR^{s \times t}$, $C_p= C_{p1}+ C_{p2}\j \in \QR^{m \times s}$, $D_p \in \QR^{n \times t}$, $E= E_1+ E_2\j \in \QR^{m \times t}$ for $l=1,2, \ldots, r$ and $p= 1, 2, \ldots, q$. Let
\begin{equation*}
	S:= \left(\sum_{l= 1}^{r}\left(h(B_l)^T \otimes A_{l1}+ h(B_l\j)^T \otimes A_{l2}\right)\right)\mathcal{W}_{ns}M_L,
\end{equation*}
\begin{equation*}
	N:= \left(\sum_{p= 1}^{q}\left(h(D_p)^T \otimes C_{p1}+ h(D_p\j)^T \otimes C_{p2}\right)\right) \mathcal{S}_{ns} \mathcal{W}_{ns}M_L.
\end{equation*}
$Q_1, Q_2$, and $e$ (as in \eqref{eq4.28}) for solving RBME of the form \eqref{eq1} are in the following form:
\begin{equation*}
	Q_1 := \Re(S) + \Re(N), \;  Q_2 :=  \Im(S)+ \Im(N), \; \mbox{and} \; 
	e := \begin{bmatrix}
		\vec(\Re(\Psi_{E}))  \\
		\vec(\Im(\Psi_{E}))
	\end{bmatrix}.
\end{equation*}
\textbf{Generalized coupled linear matrix equations in one unknown L-structure} \\
	Consider the matrix equation \eqref{eq3} and let $A_l= A_{l1}+ A_{l2}\j \in \QR^{m \times n}, B_l \in \QR^{s \times t}$, and $E_l =E_{l1}+ E_{l2}\j \in \QR^{m \times t}$ for $l= 1, 2, \ldots, r$. Let
\begin{equation*}
	T:= \begin{bmatrix}
		h(B_1)^T \otimes A_{11}+ h(B_1\j)^T \otimes A_{12} \\
		h(B_2)^T \otimes A_{21}+ h(B_2\j)^T \otimes A_{22} \\
		\vdots \\
		h(B_r)^T \otimes A_{r1}+ h(B_r\j)^T \otimes A_{r2} 
	\end{bmatrix}\mathcal{W}_{ns} M_{L}, \;  \;
	z := \begin{bmatrix}
		\vec(\Psi_{E_1}) \\
		\vec(\Psi_{E_2}) \\
		\vdots \\
		\vec(\Psi_{E_r}) 
	\end{bmatrix}.
\end{equation*}
$Q_1, Q_2$, and $e$ (as in \eqref{eq4.28}) for solving RBME of the form \eqref{eq3} are in the following form:
\begin{equation*}
	Q_1:= \Re(T), \; Q_2:= \Im(T), \; \mbox{and} \;  e:= \begin{bmatrix}
		\Re(z)  \\
		\Im(z)
	\end{bmatrix}.
\end{equation*}
	\section{Applications} \label{sec5}
	We now employ the framework developed in Section \ref{sec4} to specific cases and examine how our developed theory applies to various applications; including L-structure solutions to complex matrix equations, L-structure solutions to real matrix equations, PDIEP, and generalized PDIEP. 
\subsection{Solutions of matrix equation $AXB+CYD= E$ for $[X, Y] \in  \mathbb{H}\mathbb{C}^{n \times n} \times \mathbb{H}\mathbb{C}^{n \times n}$} \label{sec5.3}
	As a special case, we now discuss the Hankel solutions of the complex matrix equation 
	\begin{equation}\label{eq391}
		AXB+CYD= E,
	\end{equation}
\noindent
	where $A, C \in \C^{m \times n}$, $B, D \in \C^{n \times s}$, and $E \in \C^{m \times s}$. The following notations are required for solving matrix equation \eqref{eq391}. Set
		\begin{equation}\label{eq40}
		W := \left(B^T \otimes A\right) 
		\left[
			I_{n^2},   \i \, I_{n^2}
	\right]
		\begin{bmatrix}
			K_H    &       0     \\
			0         &      K_H
		\end{bmatrix},   \;
		J := \left(D^T \otimes C\right) 
	\left[
			I_{n^2},  \i \, I_{n^2}
	\right]
		\begin{bmatrix}
			K_H    &       0     \\
			0         &      K_H
		\end{bmatrix} .
	\end{equation}
Further $Q_1, Q_2, x$, and $e$ (as in \eqref{eq15}) are given in the form:
\begin{equation}\label{eq41}
	Q_1:= \left[
		\Re(W),  \Re(J)
	\right], \; 
	Q_2:= \left[
		\Im(W),  \Im(J)
	\right], \; 
	x:= \begin{bmatrix}
		\vec_H(\Re(X))   \\
		\vec_H(\Im(X))   \\
		\vec_H(\Re(Y))   \\
		\vec_H(\Im(Y))   
	\end{bmatrix}, \; \mbox{and} \; 
	e:= \begin{bmatrix}
		\vec(\Re(E)) \\
		\vec(\Im(E)) 
	\end{bmatrix}.
\end{equation}
Using  \eqref{eq13}, \eqref{eq40}, and Lemma \ref{lem3.10}, we have
\begin{eqnarray*}
   \vec(AXB) &=& (B^T \otimes A) \vec(X) \\
                     &=&  (B^T \otimes A) \left(\vec(\Re(X))+ \i \, \vec(\Im(X)) \right)\\
                     &=& (B^T \otimes A) 
                     	\left[
                     I_{n^2},  \i \, I_{n^2}
                     \right]
                     \begin{bmatrix}
                     	\vec(\Re(X)) \\
                     	\vec(\Im(X))
                     \end{bmatrix} \\
                 &=& (B^T \otimes A) 
              	\left[
              I_{n^2},   \i \, I_{n^2}
              \right]
                 \begin{bmatrix}
                 	K_H & 0         \\
                 	0      & K_H
                 \end{bmatrix}
                 \begin{bmatrix}
                 	\vec_H(\Re(X)) \\
                 	\vec_H(\Im(X))
                 \end{bmatrix} \\
             & =& W \begin{bmatrix}
             	\vec_H(\Re(X)) \\
             	\vec_H(\Im(X))
             \end{bmatrix}.
\end{eqnarray*}
Similarly,
\begin{equation*}
	\vec(CYD)= (D^T \otimes C) 
		\left[
	I_{n^2},  \i \, I_{n^2}
	\right]
	\begin{bmatrix}
		K_H & 0         \\
		0      & K_H
	\end{bmatrix}
	\begin{bmatrix}
		\vec_H(\Re(Y)) \\
		\vec_H(\Im(Y))
	\end{bmatrix} = J
	\begin{bmatrix}
		\vec_H(\Re(Y)) \\
		\vec_H(\Im(Y))
	\end{bmatrix}.
\end{equation*}
Using \eqref{eq41} and Lemma \ref{lem4.2}, we have
\begin{eqnarray*}
	AXB+CYD=E &\iff& \vec(AXB)+\vec(CYD)=\vec(E) \\
                           & \iff & W \begin{bmatrix}
                           	\vec_H(\Re(X)) \\
                           	\vec_H(\Im(X))
                           \end{bmatrix}+ J
                       \begin{bmatrix}
                       \vec_H(\Re(Y)) \\
                       \vec_H(\Im(Y))
                   \end{bmatrix} = \vec(E) \\
                   & \iff & \left[
                   	W,  J
                   \right]\begin{bmatrix}
                   \vec_H(\Re(X)) \\
                    \vec_H(\Im(X)) \\
                     \vec_H(\Re(Y)) \\
                     \vec_H(\Im(Y)) \\
               \end{bmatrix}  = \vec(E)\\
           & \iff & \begin{bmatrix}
           	\Re(W) & \Re(J) \\
           	\Im(W)  & \Im(J)
           \end{bmatrix}x = \begin{bmatrix}
           \vec(\Re(E)) \\
           \vec(\Im(E))
       \end{bmatrix} \\
   & \iff & \begin{bmatrix}
   	Q_1 \\
   	Q_2
   \end{bmatrix}x = e.
\end{eqnarray*}
Hence, matrix equation $AXB+CYD=E$ for $[X, Y] \in \mathbb{H}\mathbb{C}^{n \times n} \times \mathbb{H}\mathbb{C}^{n \times n}$ can be solved by solving the following unconstrained real matrix system: 
\begin{equation*}
 \begin{bmatrix}
 	Q_1  \\
 	Q_2
 \end{bmatrix}x= e.
\end{equation*}
 Using Lemma \ref{lem3.10}, we have
	\[
	\begin{bmatrix}
		\vec(\Re(X)) \\
		\vec(\Im(X)) \\
		\vec(\Re(Y)) \\
		\vec(\Im(Y)) 
	\end{bmatrix} = \begin{bmatrix}
	K_H & 0 & 0 & 0\\
	0      & K_H&0 & 0 \\
	0 & 0 & K_H & 0 \\
	0 & 0 & 0 & K_H
\end{bmatrix}x.
\] 

		\subsection{Solutions of matrix equation $AXB+CYD= E$ for $[X, Y] \in  \mathbb{S}\mathbb{T}\R^{n \times n} \times \mathbb{S}\mathbb{T}\R^{n \times n}$} \label{sec5.4}
	As a special case, we now discuss the symmetric Toeplitz solutions of the real matrix equation 
	\begin{equation}\label{eq39}
		AXB+CYD= E,
	\end{equation}
where $A, C \in \R^{m \times n}$, $B, D \in \R^{n \times s}$, and $E \in \R^{m \times s}$. Using Lemma \ref{lem3.91}, we have
\begin{eqnarray*}
	AXB+ CYD=E & \iff & \vec(AXB)+\vec(CYD)=\vec(E) \\
	            & \iff & \left(B^T \otimes A\right)\vec(X)+\left(D^T \otimes C\right)\vec(Y) = \vec(E) \\
	            & \iff &  \left(B^T \otimes A\right)K_{ST}\vec_{ST}(X) +  \left(D^T \otimes C\right)K_{ST}\vec_{ST}(Y) = \vec(E) \\
	            & \iff & \left[
	            	\left(B^T \otimes A\right)K_{ST},  \left(D^T \otimes C\right)K_{ST}
	            \right]\begin{bmatrix}
	               \vec_{ST}(X) \\
	               \vec_{ST}(Y) 
            \end{bmatrix} =  \vec(E).
\end{eqnarray*}
	Hence, matrix equation $AXB+CYD=E$ for $[X, Y] \in  \mathbb{S}\mathbb{T}\R^{n \times n} \times \mathbb{S}\mathbb{T}\R^{n \times n}$ can be solved by solving the following unconstrained real matrix system: 
	\begin{equation*}
		\widetilde{Q}x=\widetilde{e},
	\end{equation*}
where $\widetilde{Q}= \left[
	\left(B^T \otimes A\right)K_{ST},  \left(D^T \otimes C\right)K_{ST}
\right]$, $x=\begin{bmatrix}
\vec_{ST}(X) \\
\vec_{ST}(Y) 
\end{bmatrix}$, and $\widetilde{e}=\vec(E)$.
 Using Lemma \ref{lem3.91}, we have
\[
\begin{bmatrix}
	\vec(X) \\
	\vec(Y)
\end{bmatrix}=\begin{bmatrix}
   K_{ST} & 0 \\
   0 & K_{ST}
\end{bmatrix}x.
\]
\subsection{PDIEP and Generalized PDIEP}
	In this subsection, we aim to demonstrate the application of our developed framework in solving a range of inverse problems. Here, we develop a numerical solution methodology for the inverse problems in which the spectral constraints involve only a few eigenpair information rather than the entire spectrum. Mathematically, the problem statement is as follows: 
	\begin{problem}[PDIEP]\label{pb5.4}
		Given vectors $\{u_1,u_2,\ldots,u_k\} \subset \mathbb{F}^n$, values $\{\lambda_1, \lambda_2, \ldots, \lambda_k\} \subset \mathbb{F}$, and a set $\mathcal{L}$ of structured matrices, find a matrix $M \in \mathcal{L}$ such that
		\begin{equation*}
			Mu_i = \lambda_i u_i, \; \; \; \; i=1,2, \ldots, k,
		\end{equation*}
	where $\mathbb{F}$ represents the scalar field of either real $\R$ or complex $\C$.
	\end{problem}
	To simplify the discussion, we will use the matrix pair $\left(\Lambda, \Phi\right)$ to describe partial eigenpair information, where
\begin{equation}\label{eq45}
	\Lambda= \diag(\lambda_1,\lambda_2,\ldots,\lambda_k) \in \mathbb{F}^{k \times k} , \; \mbox{and} \; \Phi= [u_1,u_2,\ldots,u_k] \in \mathbb{F}^{n \times k} .
\end{equation}
\begin{remark}
	PDIEP can be written as $M\Phi=  \Phi \Lambda$. By using the transformations
	\begin{equation*}
		A= I_{n}, \; X= M, \; B= \Phi, \; \mbox{and} \; E= \Phi \Lambda, 
	\end{equation*}
	we can find solution to PDIEP by solving matrix equation $AXB= E$ for $X \in \mathcal{L}$.
\end{remark}
Next, we investigate generalized PDIEPs. In a nutshell, the problem is:
	\begin{problem}[Generalized PDIEP]\label{pb5.5}
	Given vectors $\{u_1,u_2,\ldots,u_k\} \subset \mathbb{F}^n$, values $\{\lambda_1, \lambda_2, \ldots, \lambda_k\} \subset \mathbb{F}$, and a set $\mathcal{L}$ of structured matrices, find pair of matrices $M, N \in \mathcal{L}$ such that
	\begin{equation*}
		Mu_i = \lambda_i N u_i, \; \; \; \; i=1,2, \ldots, k,
	\end{equation*}
	where $\mathbb{F}$ represents the scalar field of either real $\R$ or complex $\C$.
\end{problem}
\begin{remark}
	Generalized PDIEP can be written as $M\Phi=  N \Phi \Lambda$, where $\Lambda$ and $\Phi$ are as in \eqref{eq45}. By using the transformations
	\begin{equation*}
		A= I_n, \; X= M, \; B= \Phi, \; C= -I_n, \; Y= N, \; D= \Phi \Lambda, \; \mbox{and} \; E= 0, 
	\end{equation*}
	we can find solution to Generalized PDIEP by solving matrix equation $AXB+ CYD= E$ for $X, Y \in \mathcal{L}$.
\end{remark}
Though the primary emphasis of this paper is on inverse problems having symmetric Toeplitz or Hankel structures, the overall approach can be extended to encompass any structures where any set of linear relationships among matrix entries is permissible.
 \section{Numerical verification} \label{sec6}
	In this section, we present numerical examples to verify our findings. All calculations are performed on an Intel Core $i7-9700@3.00GHz/16GB$ computer using MATLAB $R2021b$ software. We now present an example to verify our method for finding the least squares solution of the RBME of the form \eqref{eq2}.
	\begin{exam} \label{ex6.1}
		Let 
		\begin{eqnarray*}
			A &=& rand(4,5)+ rand(4,5)\j, \; B= rand(5,7)+ rand(5,7)\j, \\
			C &=&  ones(4,5)+ rand(4,5)\j, \; D= rand(5,7)+ ones(5,7)\j. 
		\end{eqnarray*}
		Let $c_1= \left[\i , \, 2+\i , \, 0 , \, 1 , \, \i \right]$, $r_1= \left[\i , \, 0 , \, 2\i , \, 1  , \, 1+\i \right]$, $c_2= \left[1 , \, 3\i  , \, 2+3\i  , \, 1 , \, 0\right]$, and $r_2= \left[1 , \, 0  , \, 1 , \, \i  , \, 2\right]$. Define $$\widetilde{X}= \widetilde{X}_1+ \widetilde{X}_2\j,$$ where $\widetilde{X}_1= toeplitz(c_1, r_1)$ and $\widetilde{X}_2= toeplitz(c_2, r_2)$. 
		
		\noin
	Let $c_3= \left[2+\i , \, 4  , \, \i , \, 1+3\i  , \, 2\i \right]$, $r_3= \left[2+\i  , \, 7+6\i , \, 3+2\i , \, \i  , \, 1+\i \right]$, $c_4= \left[1+3\i , \, 3\i  , \, 2+3\i , \, 3 , \, 5+\i\right]$, and $r_4= \left[1+3\i  , \, 5  , \, 1+6\i  , \, 3+\i  , \, 2\i\right]$. Define $$\widetilde{Y}= \widetilde{Y}_1+ \widetilde{Y}_2\j,$$ where $\widetilde{Y}_1= toeplitz(c_3, r_3)$ and $\widetilde{Y}_2= toeplitz(c_4, r_4)$. Let $E= A\widetilde{X}B+C\widetilde{Y}D$. Hence, $\left[\widetilde{X}, \widetilde{Y}\right]$ is the least squares Toeplitz solution with the least norm of the RBME $AXB+CYD= E$.\\  Next, we take matrices $A, B, C, D$, and $E$ as input and apply Theorem \ref{thm4.1} to calculate the least squares Toeplitz solution with least norm of the RBME $AXB+CYD=E$. We obtain $X= X_1+ X_2\j$ and $Y= Y_1+ Y_2\j$, where
		\begin{eqnarray*}
			X_1 &=& 
		\begin{bmatrix}
			0+ 1\i   &  0 - 0\i    &  0+ 2\i   &  1- 0\i  	 &  1+ 1\i  \\
			2+ 1\i   &    0+ 1\i   &   0- 0\i   &  0+ 2\i &   1- 0\i  \\
			0- 0\i   &  2+ 1\i     &  0+ 1\i  &  0- 0\i   &  0+ 2\i   \\
			1+ 0\i    & 0- 0\i   &  2+ 1\i   &  0+ 1\i   &  	0- 0\i   \\
			0+ 1\i   &  1+ 0\i    &  0- 0\i   &  2+ 1\i    &  	0+ 1\i
		\end{bmatrix}, \;  X_2=
		\begin{bmatrix}
			1- 0\i    &  0+ 0\i   &   1- 0\i    & 0+ 1\i      &     2- 0\i   \\
			0+ 3\i   &  1- 0\i     & 0+ 0\i       & 1- 0\i   &     0+ 1\i   \\
			2+ 3\i   &   0+ 3\i    &  1- 0\i    &  0+ 0\i    &   1- 0\i   \\
			1+ 0\i   &  2+ 3\i    &  0+ 3\i     & 1- 0\i     &     0+ 0\i  \\
			0+ 0\i &  1+ 0\i      &  2+ 3\i     & 0+ 3\i     &     1- 0\i
		\end{bmatrix} , \\
		Y_1 &=& 
	\begin{bmatrix}
		2+ 1\i   &  7+6\i    &  3+ 2\i   &  0+1 \i  	 &  1+ 1\i  \\
		4+ 0\i   &   2+ 1\i   &  7+ 6\i   & 3+ 2\i &   0+ 1\i  \\
		0+ 1\i   &  4+ 0\i     &  2+ 1\i  &  7+ 6\i   &  3+ 2\i   \\
		1+ 3\i    & 0+ 1\i   &  4+ 0\i   &  2+ 1\i   &  	7+ 6\i   \\
		0+ 2\i   &  1+ 3\i    &  0+ 1\i   &  4+ 0\i    &  	2+ 1\i
	\end{bmatrix}, \;  Y_2=
	\begin{bmatrix}
		1+ 3\i    &  5+ 0\i   &   1+ 6\i    & 3+ 1\i      &     0+ 2\i   \\
		0+ 3\i   &  1+ 3\i     & 5+ 0\i       & 1+ 6\i   &     3+ 1\i   \\
		2+ 3\i   &   0+ 3\i    &  1+ 3\i    &  5+ 0\i    &   1+ 6\i   \\
		3+ 0\i   &  2+ 3\i    &  0+ 3\i     & 1+ 3\i     &     5+ 0\i  \\
		5+ 1\i &  3+ 0\i      &  2+ 3\i     & 0+ 3\i     &     1+ 3\i
	\end{bmatrix}.
\end{eqnarray*}
		Clearly, $X$ and $Y$ are reduced biquaternion Toeplitz matrices. We have $\epsilon= \norm{\left[X, Y\right]- \left[\widetilde{X}, \widetilde{Y}\right]}_F= 1. 7470 \times 10^{-13}$.
	\end{exam}
	From Example \ref{ex6.1}, we find that the error $\epsilon$ is in the order of $10^{-13}$ and is negligible. This demonstrates the effectiveness of our method in determining the structure-constrained least squares solution to the RBME of the form \eqref{eq2}. 
Next, we illustrate an example for finding the structure-constrained least squares solution to the RBME of the form \eqref{eq3}.
	\begin{exam} \label{ex6.2}
	Let 
	\begin{eqnarray*}
		A &=& ones(4,5)+ rand(4,5)\j, \; B= ones(5,7)+ rand(5,7)\j, \\
		C &=&  rand(4,5)+ rand(4,5)\j, \; D= ones(5,7)+ rand(5,7)\j. 
	\end{eqnarray*}
	Let $c_1= \left[3+\i , \, 2+4\i , \, 6+\i , \, 2+\i , \, 3\i \right]$, $r_1= \left[3\i , \, 7 , \, 3+2\i , \, 1+\i  , \, 9+\i \right]$, $c_2= \left[1+2\i  , \, 5+3\i  , \, 3\i , \, 1+7\i , \, 3\right]$, and $r_2= \left[3 , \, 1+\i , \, 2+8\i  , \, 2+\i  , \, 2+2\i\right]$. Define $$\widetilde{X}= \widetilde{X}_1+ \widetilde{X}_2\j,$$ where $\widetilde{X}_1= hankel(c_1, r_1)$ and $\widetilde{X}_2= hankel(c_2, r_2)$. Let $E= A\widetilde{X}B$ and $F=C\widetilde{X}D$. Hence, $\widetilde{X}$ is the least squares Hankel solution with the least norm of the RBME $\left(AXB, CXD\right)= \left(E, F\right)$. 
	
	Next, we take matrices $A, B, C, D, E$, and $F$ as input to calculate the least squares Hankel solution with least norm of the RBME $\left(AXB, CXD\right)= \left(E, F\right)$. We obtain $X= X_1+ X_2\j$, where
	\begin{equation*}
		X_1 = 
		\begin{bmatrix}
			3+ 1\i   &  2 + 4\i    &  6+ 1\i   &  2+ 1\i  	 &  0+ 3\i  \\
			2+ 4\i   &  6+ 1\i     &  2+ 1\i   &  0+ 3\i     &  7+ 0\i  \\
			6+ 1\i   &  2+ 1\i     &  0+ 3\i   &  7+ 0\i     &  3+ 2\i   \\
			2+ 1\i    & 0+ 3\i   &  7+ 0\i   &  3+ 2\i   &  	1+ 1\i   \\
			0+ 3\i   &  7+ 0\i    &  3+ 2\i   &  1+ 1\i    &  	9+ 1\i
		\end{bmatrix}, \;  X_2=
		\begin{bmatrix}
			1+ 2\i    &  5+ 3\i    &   0+ 3\i    & 1+ 7\i      &     3+ 0\i   \\
			5+ 3\i   &  0+ 3\i    & 1+ 7\i       & 3+ 0\i     &     1+ 1\i   \\
			0+ 3\i   &  1+ 7\i    &  3+ 0\i      & 1+ 1\i      &     2+ 8\i   \\
			1+ 7\i   &  3+ 0\i    &  1+ 1\i       & 2+ 8\i     &     2+ 1\i  \\
			3+ 0\i  &  1+ 1\i     &  2+ 8\i     & 2+ 1\i     &     2+2\i
		\end{bmatrix} . 
	\end{equation*}
	Clearly, $X$ is a reduced biquaternion Hankel matrix. We have $\epsilon= \norm{X-\widetilde{X}}_F= 5. 7042 \times 10^{-13}$.
\end{exam}
	From Example \ref{ex6.2}, we find that the error $\epsilon$ is in the order of $10^{-13}$ and is negligible. This demonstrates the effectiveness of our method in determining the structure-constrained least squares solution to the RBME of the form \eqref{eq3}.
	 
	Next, we will discuss Hankel PDIEPs \cite[Problem 5.1]{chu2005inverse}. Given a set of vectors $\left\{u_1, u_2, \ldots, u_k\right\} \subset \mathbb{C}^n$, where $k \geq 1$, and a set of  numbers $\left\{\lambda_1, \lambda_2, \ldots, \lambda_k\right\} \subset \mathbb{C}$, our aim is to construct a Hankel matrix $M \in \C^{n \times n}$ satisfying $Mu_i = \lambda_iu_i$ for $i=1,2,\ldots,k$. Now, we will illustrate this problem with an example.
	\begin{exam}\label{ex6.3}
		To establish test data, we first generate a Hankel matrix $M$. Let $c=\left[1+2\i, \, 2-4\i, \, -1+3\i, \, 4 \right]$ and $r=\left[4, \, 3+4\i, \, 2\i, \, 3 \right]$. Define $M=hankel(c,r)$.
 Let $\left(\Lambda, \Phi\right)$ denote its eigenpairs. We have $\Lambda= \diag(\lambda_1,\ldots, \lambda_4) \in \mathbb{C}^{4 \times 4}$ and $\Phi=\left[u_1, u_2, u_3, u_4\right] \in \mathbb{C}^{4 \times 4}$, where $$\left[\lambda_1, \lambda_2, \lambda_3, \lambda_4\right]=\left[-3.8029 + 7.9250\i, \, -2.7826 - 3.5629\i, \, 5.6954 - 1.0619\i, \, 6.8900 + 5.6998\i\right],$$ and their corresponding eigenvectors
			\[
		\begin{bmatrix}
			u_1 & u_2 & u_3 & u_4
			\end{bmatrix}=\begin{bmatrix}
			0.6240 + 0.0000\i & -0.4940 - 0.0377\i & -0.5395 - 0.2011\i& 	0.1572 - 0.2047\i \\
			-0.6145 - 0.0885\i& -0.5863 + 0.0219\i & 0.0172 - 0.1236\i  &  	0.4818 - 0.1113\i \\
			0.4246 + 0.0774\i &	0.1217 - 0.1368\i    &0.5855 + 0.0000\i &  	0.6784 + 0.0000\i \\
			-0.1893 + 0.0550\i &0.6138 + 0.0000\i&-0.5259 - 0.1832\i & 	0.4609 - 0.1275\i
		\end{bmatrix}.
		\]
		\textbf{Case $\mathbf{1}$.} Reconstruction from one eigenpair $(k=1)$:
		Let the prescribed partial eigeninformation be given by 
		\[
		\widetilde{\Lambda} = \lambda_3 \in \mathbb{C} \; \textrm{and} \; \; \widetilde{\Phi} = u_3 \in \mathbb{C}^{4 \times 1}.
		\]
			Construct the Hankel matrix $\widetilde{M}$ such that $\widetilde{M}u_i = \lambda_i u_i$ for $i=3$. By using the transformations $A= I_4, \; X= \widetilde{M}, \; B=\widetilde{\Phi}, \; \mbox{and} \; E= \widetilde{\Phi} \widetilde{\Lambda}$, we find the Hankel solution to the matrix equation $AXB=E$. We obtain
			\[
			\widetilde{M}=\begin{bmatrix}
				1.6614 + 0.3115\i & 1.0564 + 0.6597\i & -1.8088 + 0.4921\i &  2.6736 - 0.4763\i \\
				1.0564 + 0.6597\i & -1.8088 + 0.4921\i &  2.6736 - 0.4763\i &  2.0823 - 0.5222\i \\
			-1.8088 + 0.4921\i &  2.6736 - 0.4763\i &  2.0823 - 0.5222\i & -1.7415 + 0.7505\i \\
				2.6736 - 0.4763\i &  2.0823 - 0.5222\i & -1.7415 + 0.7505\i &  1.2459 + 0.2833\i
			\end{bmatrix}.
			\]
				Then, $\widetilde{M}$ is the desired Hankel matrix. 
				
				\noin
		\textbf{Case $\mathbf{2}$.} Reconstruction from two eigenpairs $(k=2)$:
		Let the prescribed partial eigeninformation be given by 
		\[
		\widetilde{\Lambda} = \diag(\lambda_2, \lambda_3) \in \mathbb{C}^{2 \times 2} \; \textrm{and} \; \; \widetilde{\Phi} = \left[u_2, u_3\right]\in \mathbb{C}^{4 \times 2}.
		\]
	Construct the Hankel matrix $\widetilde{M}$ such that $\widetilde{M}u_i = \lambda_i u_i$ for $i=2,3$. By using the transformations $A= I_4, \; X= \widetilde{M}, \; B=\widetilde{\Phi}, \; \mbox{and} \; E= \widetilde{\Phi} \widetilde{\Lambda}$, we find the Hankel solution to the matrix equation $AXB=E$. We obtain
	\[\widetilde{M}= \begin{bmatrix}
			1.0000 + 2.0000\i   & 2.0000 - 4.0000\i  & -1.0000 + 3.0000\i &  4.0000 - 0.0000\i \\
			2.0000 - 4.0000\i  & -1.0000 + 3.0000\i  & 4.0000 - 0.0000\i  & 3.0000 + 4.0000\i \\
			-1.0000 + 3.0000\i &  4.0000 - 0.0000\i  & 3.0000 + 4.0000\i  &-0.0000 + 2.0000\i \\
			4.0000 - 0.0000\i   & 3.0000 + 4.0000\i  & -0.0000 + 2.0000\i &  3.0000 + 0.0000\i 
		\end{bmatrix}.
	\]
	Then, $\widetilde{M}$ is the desired Hankel matrix. 
		\begin{table}[H]
		\centering
		\begin{tabular}{cccc}
			\toprule%
			\multicolumn{2}{c}{\textbf{Case $\mathbf{1}$ $(k=1)$}} & \multicolumn{2}{c}{\textbf{Case $\mathbf{2}$ $(k=2)$}} \\
			\cmidrule(lr){1-2}  \cmidrule(lr){3-4}%
			Eigenpair & Residual $\norm{\widetilde{M} u_i- \lambda_i u_i}_2$ & Eigenpairs & Residual $\norm{\widetilde{M} u_i- \lambda_i u_i}_2$ \\
			\midrule
			$\left(\lambda_3, u_3\right)$& $2.7792 \times 10^{-15}$  & $\left(\lambda_2, u_2\right)$   & $3.1349 \times 10^{-14}$  \\
			&  & $\left(\lambda_3, u_3\right)$  & $2.2761 \times 10^{-14}$\\ 
			\bottomrule
		\end{tabular}
		\caption{Residual $\norm{\widetilde{M} u_i- \lambda_i u_i}_2$ for Example \ref{ex6.3}} \label{tab1}
	\end{table}
	\end{exam}
From Table \ref{tab1}, we find that the residual $\norm{\widetilde{M} u_i- \lambda_i u_i}_2$, for $i=3$ in Case $1$ and for $i=2, 3$ in Case $2$, is in the order of $10^{-14}$ and is negligible. This demonstrates the effectiveness of our method in solving the Hankel PDIEP.

Next, we will discuss symmetric Toeplitz PDIEPs \cite[Problem 5.2]{chu2005inverse}. Given a set of real orthonormal vectors $\left\{u_1, u_2, \ldots, u_k\right\} \subset \mathbb{R}^n$, where $k \geq 1$, each symmetric or skew-symmetric, and a set of real numbers $\left\{\lambda_1, \lambda_2, \ldots, \lambda_k\right\} \subset \mathbb{R}$, our aim is to construct a symmetric Toeplitz matrix $T \in \R^{n \times n}$ satisfying $Tu_i = \lambda_iu_i$ for $i=1,2,\ldots,k$. It is important to note that a vector $u \in \R^n$ is called symmetric if $Ju=u$ and skew-symmetric if $Ju=-u$, where $J$ is the exchange matrix (a square matrix with ones on the anti-diagonal and zeros elsewhere). Now we illustrate this problem with an example.
		\begin{exam}\label{ex6.4}
		To establish test data, we first generate a real symmetric Toeplitz matrix $T$. Let $c=\left[5.30, \, 2.50, \, 4.60, \, -3.70,\, 2.80\right]$. Define $T=toeplitz(c)$. Let $\left(\Lambda, \Phi\right)$ denote its eigenpairs. We have $\Lambda= \diag(\lambda_1,\ldots, \lambda_5) \in \mathbb{R}^{5 \times 5}$ and $\Phi=\left[u_1, u_2, u_3, u_4, u_5\right] \in \mathbb{R}^{5 \times 5}$, where 
		$$\left[\lambda_1, \lambda_2, \lambda_3, \lambda_4, \lambda_5\right]=\left[-4.6650, \, -1.0842, \, 7.8650, \, 10.4951, \, 13.8891\right],$$ and their corresponding eigenvectors\\ 
		\[
		\begin{bmatrix}
			u_1& u_2& u_3& u_4& u_5
			\end{bmatrix}=\begin{bmatrix}
			0.4627 & 0.4077 & 0.5347 & -0.3460 & -0.4627 \\
			-0.5347 & 0.2169 & 0.4627 & 0.6165 & -0.2699 \\
			-0.0000 & -0.7573& 0.0000 & -0.0193 & -0.6528 \\
			0.5347 & 0.2169 & -0.4627 & 0.6165 & -0.2699 \\
			-0.4627 & 0.4077 & -0.5347 & 0.3460 & -0.4627
		\end{bmatrix}.
		\]
		\textbf{Case $\mathbf{1}$.} Reconstruction from two eigenpairs in which one eigenvector is odd and other is even $(k=2)$:
			Let the prescribed partial eigeninformation be given by 
		\[
		\widetilde{\Lambda} = \diag(\lambda_1, \lambda_2) \in \mathbb{R}^{2 \times 2} \; \textrm{and} \; \; \widetilde{\Phi} = \left[u_1, u_2\right]\in \mathbb{R}^{5 \times 2}.
		\]
		Construct the symmetric Toeplitz matrix $\widetilde{T}$ such that $\widetilde{T}u_i = \lambda_i u_i$ for $i=1,2$. By using the transformations $A= I_5, \; X= \widetilde{T}, \; B=\widetilde{\Phi}, \; \mbox{and} \; E= \widetilde{\Phi} \widetilde{\Lambda}$, we find the symmetric Toeplitz solution to the matrix equation $AXB=E$. We obtain
			$$ \widetilde{T}= \begin{bmatrix}
				5.30 & 2.50 & 4.60 & -3.70 & 2.80 \\
				2.50 & 5.30 & 2.50 & 4.60 & -3.70 \\
				4.60 & 2.50 & 5.30 & 2.50 & 4.60  \\
				-3.70 & 4.60 & 2.50 & 5.30 & 2.50 \\
				2.80 & -3.70 & 4.60 & 2.50 & 5.30
			\end{bmatrix}.$$
		Then, $\widetilde{T}$ is the desired symmetric Toeplitz matrix. 
		
	\noin
		\textbf{Case $\mathbf{2}$.} Reconstruction from two eigenpairs in which both eigenvectors are odd $(k=2)$: 	Let the prescribed partial eigeninformation be given by 
		\[
	\widetilde{\Lambda} = \diag(\lambda_1, \lambda_3) \in \mathbb{R}^{2 \times 2} \; \textrm{and} \; \; \widetilde{\Phi} = \left[u_1, u_3\right]\in \mathbb{R}^{5 \times 2}.
		\]
		Construct the symmetric Toeplitz matrix $\widetilde{T}$ such that $\widetilde{T}u_i = \lambda_i u_i$ for $i=1,3$. By using the transformations $A= I_5, \; X= \widetilde{T}, \; B=\widetilde{\Phi}, \; \mbox{and} \; E= \widetilde{\Phi} \widetilde{\Lambda}$, we find the symmetric Toeplitz solution to the matrix equation $AXB=E$. We obtain
		$$\; \; \; \widetilde{T}= \begin{bmatrix}
			1.0667 & 3.1000 & 0.3667 & -3.1000 & -1.4333 \\
			3.1000 & 1.0667 & 3.1000 & 0.3667 & -3.1000 \\
			0.3667 & 3.1000 & 1.0667 & 3.1000 & 0.3667  \\
			-3.1000 & 0.3667 & 3.1000 & 1.0667 & 3.1000 \\
			-1.4333 & -3.1000 & 0.3667 & 3.1000 & 1.0667
		\end{bmatrix}.$$
	Then, $\widetilde{T}$ is the desired symmetric Toeplitz matrix. 
		\begin{table}[H]
		\centering
		\begin{tabular}{cccc}
			\toprule%
			\multicolumn{2}{c}{\textbf{Case $\mathbf{1}$ $(k=2)$}} & \multicolumn{2}{c}{\textbf{Case $\mathbf{2}$ $(k=2)$}} \\
			\cmidrule(lr){1-2}  \cmidrule(lr){3-4}%
			Eigenpairs & Residual $\norm{\widetilde{T} u_i- \lambda_i u_i}_2$ & Eigenpairs & Residual $\norm{\widetilde{T} u_i- \lambda_i u_i}_2$ \\
			\midrule
			$\left(\lambda_1, u_1\right)$& $5.7430 \times 10^{-15}$  & $\left(\lambda_1, u_1\right)$   & $2.2505 \times 10^{-15}$  \\
			$\left(\lambda_2, u_2\right)$& $1.2200 \times 10^{-14}$ & $\left(\lambda_3, u_3\right)$  & $6.1218 \times 10^{-15}$\\ 
			\bottomrule
		\end{tabular}
		\caption{Residual $\norm{\widetilde{T} u_i- \lambda_i u_i}_2$ for Example \ref{ex6.4}} \label{tab2}
	\end{table}
	\end{exam}
From Table \ref{tab2}, we find that the residual $\norm{\widetilde{T} u_i- \lambda_i u_i}_2$, for $i=1,2$ in Case $1$ and for $i=1,3$ in Case $2$, is in the order of $10^{-14}$ and is negligible. This demonstrates the effectiveness of our method in solving the symmetric Toeplitz PDIEP.

Similar to PDIEP, one can solve the generalized PDIEP. We illustrate the generalized PDIEP for Hankel structure by the following example.
\begin{exam}\label{ex6.5}
		To establish test data, we first generate a linear matrix pencil $M-\lambda N$, where $M$ and $N$ are Hankel matrices. Let $c_1=\left[
			4+2\i, \,  2-4\i, \, -1+3\i, \, 4+3\i \right]$ and $r_1= \left[4+3\i, \, 4\i, \, 9+2\i, \, 3+\i\right]$. Define $M=hankel(c_1,r_1)$. Let $c_2=\left[3+2\i, \, 6-\i, \, -5+2\i, \, 4+7\i \right]$ and $r_2=\left[4+7\i, \, 3+4\i, \, 2+2\i, \, 3-8\i \right]$. Define $N=hankel(c_2,r_2)$.  Let $\left(\Lambda, \Phi\right)$ denote its eigenpairs. We have $\Lambda= \diag(\lambda_1,\ldots, \lambda_4) \in \mathbb{C}^{4 \times 4}$ and $\Phi=\left[u_1, u_2, u_3, u_4\right] \in \mathbb{C}^{4 \times 4}$, where $$\left[\lambda_1, \lambda_2, \lambda_3, \lambda_4\right]=\left[-0.3953+0.6027\i, \, 0.3708-0.7155\i, \, 0.6743-0.3655\i, \, 0.6761+0.1157\i \right],$$ and their corresponding eigenvectors
			\[ \begin{bmatrix}
				u_1& u_2& u_3& u_4
				\end{bmatrix} =
			\begin{bmatrix}
				-0.4881+0.1767\i & -0.4811 - 0.3552\i & -0.7739 + 0.1499\i & 0.7130 + 0.2870\i \\
				0.4383 + 0.4624\i&   0.4236 + 0.5764\i &-0.8976 + 0.1024\i&  0.1416 + 0.5177\i \\
			 0.4194 - 0.5806\i & -0.1700 + 0.0352\i & -0.3007 + 0.3084\i &  -0.3339 + 0.5007\i \\
			-0.5678 - 0.0875\i &  0.3392 - 0.1123\i & 0.0061 + 0.1882\i & -0.3560 - 0.2370\i 
			\end{bmatrix}.
		\]
			\textbf{Case $\mathbf{1}$.} Reconstruction from one eigenpair $(k=1)$:
			Let the prescribed partial eigeninformation be given by 
			\[
		\widetilde{\Lambda} = \lambda_1 \in \mathbb{C} \; \textrm{and} \; \; \widetilde{\Phi} = u_1 \in \mathbb{C}^{4 \times 1}.
			\]
		Construct the Hankel matrices $\widetilde{M}$ and $\widetilde{N}$ such that $\widetilde{M}u_i = \lambda_i \widetilde{N} u_i$ for $i=1$.	By using the transformations
		$A= I_4$, $X= \widetilde{M}$, $B= \widetilde{\Phi}$, $C= -I_4$, $Y= \widetilde{N}$, $D= \widetilde{\Phi} \widetilde{\Lambda}$, and $E= 0$, we find the Hankel solution to the matrix equation $AXB+CYD=E$. We obtain
		\begin{eqnarray*}
		\widetilde{M}&=& \begin{bmatrix}
		 1.0472 + 0.3406\i &   1.1937 + 0.5288\i & 0.8984 + 0.8802\i & 1.0875 + 1.1282\i  \\
		1.1937 + 0.5288\i &  0.8984 + 0.8802\i  & 1.0875 + 1.1282\i & 0.7748 + 1.0806\i   \\
		0.8984 + 0.8802\i & 1.0875 + 1.1282\i & 0.7748 + 1.0806\i & 0.6237 + 1.3399\i \\
		1.0875 + 1.1282\i & 0.7748 + 1.0806\i & 0.6237 + 1.3399\i & 0.3267 + 1.2860\i  
		\end{bmatrix}, \\
	\widetilde{N} &=& \begin{bmatrix}
		  1.4161 + 0.7678\i & 1.3606 + 0.9305\i & 1.0320 + 0.8914\i & 0.9574 + 1.1034\i   \\
		1.3606 + 0.9305\i & 1.0320 + 0.8914\i & 0.9574 + 1.1034\i & 0.8624 + 0.8961\i \\
		1.0320 + 0.8914\i & 0.9574 + 1.1034\i & 0.8624 + 0.8961\i & 0.6464 + 0.9076\i \\
		0.9574 + 1.1034\i & 0.8624 + 0.8961\i & 0.6464 + 0.9076\i & 0.5615 + 0.7072\i
	\end{bmatrix}.
		\end{eqnarray*}
		Then, $\widetilde{M}-\lambda \widetilde{N}$ is the desired Hankel matrix pencil. 
		
		\noin
			\textbf{Case $\mathbf{2}$.} Reconstruction from two eigenpairs $(k=2)$:
		Let the prescribed partial eigeninformation be given by 
		\[
		\widetilde{\Lambda} = \diag(\lambda_1, \lambda_3) \in \mathbb{C}^{2 \times 2} \; \textrm{and} \; \; \widetilde{\Phi} = \left[u_1, u_3\right]\in \mathbb{C}^{4 \times 2}.
		\]
		Construct the Hankel matrices $\widetilde{M}$ and $\widetilde{N}$ such that $\widetilde{M}u_i = \lambda_i \widetilde{N} u_i$ for $i=1, 3$.	By using the transformations
		$A= I_4$, $X= \widetilde{M}$, $B= \widetilde{\Phi}$, $C= -I_4$, $Y= \widetilde{N}$, $D= \widetilde{\Phi} \widetilde{\Lambda}$, and $E= 0$, we find the Hankel solution to the matrix equation $AXB+CYD=E$. We obtain
		\begin{eqnarray*}
			\widetilde{M} &= & \begin{bmatrix}
				  0.2460-0.0000\i & -0.0696-0.0231\i &  0.1118-0.0226\i & -0.0519+0.0436\i \\
				-0.0696-0.0231\i & 0.1118-0.0226\i & -0.0519+0.0436\i &  0.0299+0.1325\i \\
				0.1118-0.0226\i & -0.0519+0.0436\i &  0.0299+0.1325\i &  0.1243-0.0621\i \\
				-0.0519+0.0436\i & 0.0299+0.1325\i &  0.1243-0.0621\i &  0.0711+0.0777\i
			\end{bmatrix}, \\
		\widetilde{N} &= & \begin{bmatrix}
			   0.1767+0.0416\i & 0.1067-0.0146\i  & -0.0352+0.0850\i & 0.0696-0.0910\i  \\
			   0.1067-0.0146\i & -0.0352+0.0850\i & 0.0696-0.0910\i & -0.0943+0.1694\i \\
			-0.0352+0.0850\i & 0.0696-0.0910\i & -0.0943+0.1694\i & -0.0396+0.0850\i \\
			0.0696-0.0910\i & -0.0943+0.1694\i & -0.0396+0.0850\i & -0.0269+0.0197\i
		\end{bmatrix}.
		\end{eqnarray*}
		Then, $\widetilde{M}-\lambda \widetilde{N}$ is the desired Hankel matrix pencil. 
			\begin{table}[H]
			\centering
			\begin{tabular}{cccc}
				\toprule%
				\multicolumn{2}{c}{\textbf{Case $\mathbf{1}$ $(k=1)$}} & \multicolumn{2}{c}{\textbf{Case $\mathbf{2}$ $(k=2)$}} \\
				\cmidrule(lr){1-2}  \cmidrule(lr){3-4}%
				Eigenpair & Residual $\norm{\widetilde{M} u_i- \lambda_i\widetilde{N} u_i}_2$ & Eigenpairs & Residual $\norm{\widetilde{M} u_i- \lambda_i\widetilde{N} u_i}_2$ \\
				\midrule
				$\left(\lambda_1, u_1\right)$& $2.7626 \times 10^{-15}$  & $\left(\lambda_1, u_1\right)$   & $1.0906 \times 10^{-14}$  \\
				    &  & $\left(\lambda_3, u_3\right)$  & $2.7570 \times 10^{-15}$\\ 
				\bottomrule
			\end{tabular}
			\caption{Residual $\norm{\widetilde{M} u_i- \lambda_i\widetilde{N} u_i}_2$  for Example \ref{ex6.5}} \label{tab3}
		\end{table}
	\end{exam}
From Table \ref{tab3}, we find that the residual $\norm{\widetilde{M} u_i- \lambda_i\widetilde{N} u_i}_2$, for $i=1$ in Case $1$ and for $i=1,3$ in Case $2$, is in the order of $10^{-14}$ and is negligible. This demonstrates the effectiveness of our method in solving the generalized PDIEP for Hankel structure.

Next, we will illustrate an example of generalized PDIEP for symmetric Toeplitz structure.
\begin{exam}\label{ex6.6}
		To establish test data, we first generate a linear matrix pencil $M-\lambda N$, where $M$ and $N$ are symmetric Toeplitz matrices. Let $c_1=\left[7.80, \, 5.50, \, 3.70, \, -2.30, \, 8.90 \right]$. Define $M=toeplitz(c_1)$. Let $c_2=\left[4.20, \, 1.20, \, -3.50, \, 3.90, \, 9.80 \right]$. Define $N=toeplitz(c_2)$. Let $\left(\Lambda, \Phi\right)$ denote its eigenpairs. We have $\Lambda= \diag(\lambda_1,\ldots, \lambda_5) \in \mathbb{C}^{5 \times 5}$ and $\Phi=\left[u_1, u_2, u_3, u_4, u_5\right] \in \mathbb{C}^{5 \times 5}$, where $$\left[\lambda_1, \lambda_2, \lambda_3, \lambda_4, \lambda_5\right]=\left[4.1157, \, -1.7144, \, 0.2371, \, -0.1060 + 1.1336\i, \, -0.1060 - 1.1336\i \right],$$ and their corresponding eigenvectors
		\[ \begin{bmatrix}
			u_1& u_2& u_3& u_4& u_5
			\end{bmatrix}=\begin{bmatrix}
			-0.2481 &  0.3192  & -0.2773 &  -0.2700+0.7300\i & -0.2700 - 0.7300\i \\
			-0.4470 & -0.8953 & -0.4115 & 0.6140+0.1425\i  &   0.6140 - 0.1425\i  \\
			-1.0000  & 1.0000  & 1.0000  & -0.0000-0.0000\i & -0.0000 + 0.0000\i \\
			-0.4470  & -0.8953 & -0.4115 & -0.6140-0.1425\i & -0.6140 + 0.1425\i \\
			-0.2481  & 0.3192  & -0.2773  & 0.2700-0.7300\i & 0.2700 + 0.7300\i
		\end{bmatrix}.
		\]
		\textbf{Case $\mathbf{1}$.} Reconstruction from two eigenpairs $(k=2)$:
	Let the prescribed partial eigeninformation be given by 
	\[
\widetilde{\Lambda} = \diag(\lambda_1, \lambda_3) \in \mathbb{C}^{2 \times 2} \; \textrm{and} \; \; \widetilde{\Phi} = \left[u_1, u_3\right]\in \mathbb{C}^{5 \times 2}.
	\]
	Construct the symmetric Toeplitz matrices $\widetilde{M}$ and $\widetilde{N}$ such that $\widetilde{M}u_i = \lambda_i \widetilde{N} u_i$ for $i=1, 3$. By using the transformations
	$A= I_5$, $X= \widetilde{M}$, $B= \widetilde{\Phi}$, $C= -I_5$, $Y= \widetilde{N}$, $D= \widetilde{\Phi} \widetilde{\Lambda}$, and $E= 0$, we find the symmetric Toeplitz solution to the matrix equation $AXB+CYD=E$. We obtain
	\begin{eqnarray*}
		\widetilde{M} &=& \begin{bmatrix}
			1.3921&    1.0473 &   0.6772 &  -0.2032 &   0.6735 \\
			1.0473  &  1.3921  &  1.0473  &  0.6772   &-0.2032 \\
			0.6772   & 1.0473  &  1.3921   & 1.0473   & 0.6772 \\
			-0.2032  &  0.6772 &   1.0473 &   1.3921  &  1.0473 \\
			0.6735  & -0.2032  &  0.6772   & 1.0473   & 1.3921
		\end{bmatrix}, \\
	\widetilde{N} &=& \begin{bmatrix}
		0.6339 &   0.1905 &  -0.3161&    0.6055 &   0.7404  \\
		0.1905   & 0.6339   & 0.1905  & -0.3161  &  0.6055 \\
		-0.3161   & 0.1905   & 0.6339  &  0.1905 &  -0.3161 \\
		0.6055  & -0.3161   & 0.1905  &  0.6339  &  0.1905 \\
		0.7404   & 0.6055   &-0.3161  &  0.1905   & 0.6339
	\end{bmatrix}.
	\end{eqnarray*}
	Then, $\widetilde{M}-\lambda \widetilde{N}$ is the desired symmetric Toeplitz matrix pencil. 
	
		\noin
	\textbf{Case $\mathbf{2}$.} Reconstruction from three eigenpairs $(k=3)$:
	Let the prescribed partial eigeninformation be given by 
	\[
	\widetilde{\Lambda} = \diag(\lambda_1, \lambda_2, \lambda_3) \in \mathbb{C}^{3 \times 3} \; \textrm{and} \; \; \widetilde{\Phi} = \left[u_1, u_2, u_3\right]\in \mathbb{C}^{5 \times 3}.
	\]
	Construct the symmetric Toeplitz matrices $\widetilde{M}$ and $\widetilde{N}$ such that $\widetilde{M}u_i = \lambda_i \widetilde{N} u_i$ for $i=1, 2, 3$. By using the transformations
	$A= I_5$, $X= \widetilde{M}$, $B= \widetilde{\Phi}$, $C= -I_5$, $Y= \widetilde{N}$, $D= \widetilde{\Phi} \widetilde{\Lambda}$, and $E= 0$, we find the symmetric Toeplitz solution to the matrix equation $AXB+CYD=E$. We obtain
	\begin{eqnarray*}
		\widetilde{M} &=& \begin{bmatrix}
			  0.9214 &   0.6497 &    0.4371 &  -0.2717 &   1.0513 \\
			0.6497  &  0.9214  &  0.6497   & 0.4371  & -0.2717 \\
			0.4371  &  0.6497  &  0.9214   & 0.6497   & 0.4371 \\
			-0.2717  &  0.4371  &  0.6497   & 0.9214   & 0.6497 \\
			1.0513  & -0.2717  &  0.4371   & 0.6497   & 0.9214
		\end{bmatrix}, \\
	\widetilde{N} &=& \begin{bmatrix}
		 0.4961 &   0.1417 &   -0.4134  &  0.4607  &  1.1576 \\
		0.1417   & 0.4961   & 0.1417  & -0.4134   & 0.4607  \\
		-0.4134  &  0.1417  &  0.4961  &  0.1417  & -0.4134 \\
		0.4607  & -0.4134   & 0.1417   & 0.4961   & 0.1417 \\
		1.1576   & 0.4607  & -0.4134   & 0.1417   & 0.4961
		\end{bmatrix}.
	\end{eqnarray*}
	Then, $\widetilde{M}-\lambda \widetilde{N}$ is the desired symmetric Toeplitz matrix pencil. 
\begin{table}[H]
	\centering
	\begin{tabular}{cccc}
		\toprule%
		\multicolumn{2}{c}{\textbf{Case $\mathbf{1}$ $(k=2)$}} & \multicolumn{2}{c}{\textbf{Case $\mathbf{2}$ $(k=3)$}} \\
		\cmidrule(lr){1-2}  \cmidrule(lr){3-4}%
		Eigenpairs & Residual $\norm{\widetilde{M} u_i- \lambda_i\widetilde{N} u_i}_2$ & Eigenpairs & Residual $\norm{\widetilde{M} u_i- \lambda_i\widetilde{N} u_i}_2$ \\
		\midrule
		$\left(\lambda_1, u_1\right)$& $3.3675 \times 10^{-15}$  & $\left(\lambda_1, u_1\right)$   & $6.9900 \times 10^{-15}$  \\
		$\left(\lambda_3, u_3\right)$& $2.3481 \times 10^{-15}$ & $\left(\lambda_2, u_2\right)$  & $2.4962 \times 10^{-15}$\\ 
		& & $\left(\lambda_3, u_3\right)$& $2.5686 \times 10^{-15}$ \\
		\bottomrule
	\end{tabular}
	\caption{Residual $\norm{\widetilde{M} u_i- \lambda_i\widetilde{N} u_i}_2$ for Example \ref{ex6.6}} \label{tab4}
\end{table}
\end{exam}
From Table \ref{tab4}, we find that the residual $\norm{\widetilde{M} u_i- \lambda_i\widetilde{N} u_i}_2$, for $i=1,3$ in Case $1$ and for $i=1,2,3$ in Case $2$, is in the order of $10^{-15}$ and is negligible. This demonstrates the effectiveness of our method in solving the generalized PDIEP for symmetric Toeplitz structure.
	\section{Conclusions} \label{sec7}
	In this manuscript, we have examined several L-structure reduced biquaternion matrix sets, including reduced biquaternion Toeplitz, symmetric Toeplitz, Hankel, and circulant matrix sets. Next, we have proposed a generalized framework for finding the least squares L-structure solutions for the following RBMEs:
	\begin{eqnarray*}
			&&	\sum_{l=1}^{r}A_lX_lB_l= E,     \\   
	&&	\sum_{l=1}^{r}A_lXB_l+ \sum_{p=1}^{q}C_pX^TD_p = E,    \\
      &&	(A_1XB_1, A_2XB_2, \ldots, A_rXB_r)= (E_1, E_2, \ldots, E_r).           
\end{eqnarray*}         
	Lastly, we have discussed how our developed theory applies to various applications, including L-structure solutions to complex and real matrix equations, PDIEP, and generalized PDIEP. 
	\bibliography{p1new}

\end{document}